\documentclass[11pt]{article}
\usepackage{relsize}
\usepackage{exscale}
\usepackage{authblk}
\usepackage[utf8]{inputenc}
\usepackage[%
  style=numeric,
  backend=biber,
  url=false,
  doi=false,
  isbn=false,
  eprint=false,
  giveninits=true
]{biblatex}
\NewCommandCopy{\biblatexcite}{\cite}
\RenewDocumentCommand{\cite}{oms}{%
  \IfNoValueTF{#1}{%
    \IfBooleanTF{#3}{\optcite{#2}}{\biblatexcite{#2}}%
  }{%
    \biblatexcite[#1]{#2}%
  }%
}
\NewDocumentCommand{\optcite}{mm}{\biblatexcite[#2]{#1}}

\usepackage{amsmath,amsthm,amsfonts,amssymb}
\usepackage{mathtools}
\usepackage[theoremfont]{newtxtext} %
\usepackage{newtxmath} %
\usepackage{graphicx}
\usepackage{verbatim}
\usepackage{enumitem}
 \setlist[enumerate,1]{label=(\arabic*)}
 \setlist[enumerate,2]{label=(\alph*)}
\usepackage{amssymb}
\usepackage{esint}
\usepackage{verbatim}
\usepackage[version=4]{mhchem}
\usepackage{stmaryrd}
\usepackage{color}
\usepackage{calc}
\definecolor{darkred}{rgb}{0.8,0,0}
\definecolor{darkblue}{rgb}{0,0,0.55}
\definecolor{darkgreen}{rgb}{0,0.39,0}

\newtheorem{theorem}{Theorem}[section]
\newtheorem{proposition}[theorem]{Proposition}
\newtheorem{lemma}[theorem]{Lemma}
\newtheorem{remark}[theorem]{Remark}
\newtheorem*{remark*}{Remark}
\newtheorem{corollary}[theorem]{Corollary}

\newlist{steps}{enumerate}{1}
\setlist[steps]{label=\textit{Step \arabic*:},leftmargin=0pt,itemindent={\widthof{\textit{Step 3: }}}}

\numberwithin{equation}{section}

\newcommand{\eps}{\varepsilon}

\newcommand{\R}{\mathbb{R}}
\newcommand{\N}{\mathbb{N}}

\newcommand{\calH}{\mathcal{H}}
\newcommand{\calL}{\mathcal{L}}
\newcommand{\calM}{\mathcal{M}}

\newcommand{\embeds}{\hookrightarrow}

\newcommand{\Chi}{\mathcal{X}}

\newcommand{\weakto}{\rightharpoonup}
\newcommand{\weakstarto}{\stackrel{\ast}{\rightharpoonup}}
\newcommand{\gmax}{\mathbf{g}_{\max}}
\newcommand{\gmin}{\mathbf{g}_{\min}}

\newcommand{\gmaxt}{\gmax}
\newcommand{\gbarmax}{\bar{\mathbf{g}}_{\max}}
\newcommand{\gbarmin}{\bar{\mathbf{g}}_{\min}}
\newcommand{\dd}{\,\mathrm{d}}

\DeclareMathOperator{\loc}{loc}

\DeclareMathOperator{\spt}{spt}

\DeclareMathOperator{\argmax}{arg\,max}

\usepackage{hyperref}
\begin{document}
\title{ 
Localization properties of a free boundary problem for cell polarization}
\author[a,1]{B. Niethammer}
\author[b,3] {M. R\"oger}
\author[a,2] {J.J.L. Vel\'azquez}
\affil[a]{Institute for Applied Mathematics, University of Bonn, Endenicher Allee 60, 53115 Bonn, Germany}
\affil[b]{ Technische Universität Dortmund, 
Fakultät für Mathematik, Vogelpothsweg 87,
44227 Dortmund, Germany}

\affil[1]{\href{mailto:niethammer@iam.uni-bonn.de}{niethammer@iam.uni-bonn.de}}
\affil[2]{\href{mailto:velazquez@iam.uni-bonn.de}{velazquez@iam.uni-bonn.de}}
\affil[3]{\href{mailto:roeger@tu-dortmund.de}{roeger@tu-dortmund.de}}

 \maketitle

\begin{abstract}
We investigate localization in a time-dependent contrained obstacle problem for the density of an active protein on a cell membrane. This free boundary problem has been derived in \cite{LNRV21} as a model for cell polarization under the influence of an external signal.  
We show that in the limit of small mass  there is instantaneous localization, that is  the mass of the active protein concentrates  in the set where the signal is maximal for any positive time $t>0$. 
We  give a precise a characterization of this concentration process on the appropriate slow time scale and illustrate this for a few examples of the external signal.

\\[2ex]\noindent%

\bigskip
\begin{center}
{\it \large In memory of Bob Kohn, who inspired so many with his ability to connect deep insights into applications with rigorous mathematics.}
\end{center}

\end{abstract}

\section{Introduction}

Cell polarization --- the establishment of an asymmetric spatial distribution of molecular components on the cell membrane and in the cell interior --- is a central mechanism in processes such as cell differentiation and cell motility. 
Some cell types (e.g. neurons, muscle cells) are permanently polarized, while many others (e.g. Dictyostelium discoideum, keratinocytes) polarize transiently in response to external cues. 
Different modules contribute to the cell polarization process: signal reception and input encoding converts an extracellular cue into an intracellular input signal. An adaptation module identifies relative or directional information and rejects uniform background signals. Signal amplification increases the contrast of a directional signal while preserving rapid and reversible tracking of the external cue.  Membrane localization and signal sharpening concentrate the amplified signal to a restricted membrane region and polarity selection and stabilization creates persistent patterns or front--back states.

Mathematically, cell polarization has often been modelled %
by reaction--diffusion systems that combine short--range activation with longer--range inhibition \cite{RaEd17}.
Prominent examples are the seminal work of Turing \cite{Turi52} on symmetry breaking and spontaneous polarization, its generalization in form of the LALI (local activation--lateral inhibition) paradigm of Gierer--Meinhardt \cite{GiMe72}, the LEGI (local excitation--global inhibition) model for adaptation and directional sensing \cite{PaDe99,LeIg02}, or wave--pinning models that create persistent polarized domains \cite{MoJE08}.

In realistic cells membrane reactions are coupled to attachment--detachment kinetics and diffusion in the cell interior; this coupling motivates bulk--surface models that treat membrane concentrations and cytosolic concentrations as coupled unknowns. 
Such bulk--surface formulations admit additional symmetry--breaking mechanisms and have been the subject of intensive study \cite{LeRa05,NGCR07,RaeRoe12,RaeRoe14}, also increasingly from a rigorous PDE viewpoint \cite{Miel13,BKMS17,FeLT18,HaRo18,AlET18,MoTa23,AuBo24,Diss21}.

For the mathematical analysis of amplification and pattern selection, asymptotic reductions are particularly useful. 
In certain parameter regimes one can rigorously (or formally) reduce bulk--surface systems to more accessible limit problems, notably free--boundary problems on the membrane \cite{ElRV17,DMMS18,MFNS23,GKRR16,AbKa20,KoSU26}.

\bigskip
In this paper we continue the study of a minimal model for the signal amplification and localization steps in cell polarization developed in \cite{NiRV20,LNRV21,LNRV23,LNRV24}. 

To describe our results in more detail, let us first review the models studied in \cite{NiRV20,LNRV21}. 
The starting point is a coupled bulk-surface reaction--diffusion system for the densities of the  active and the inactive version of a protein respectively,
where the deactivation term is given by a Michaelis--Menten law and the other reaction and attachment/detachment laws are linear.
In the limit of large rate constants and large total protein mass
this system reduces to a free boundary problem for the concentration of the active protein on the membrane, in the following denoted by $u$, coupled to the concentration of the inactive protein $w$ in the interior of the cell.  

In the following $\Omega \subset \R^3$ is a bounded connected domain which represents the cell and $\Gamma:=\partial \Omega$ is a smooth closed surface representing the membrane of the cell. 
We assume that the external chemical signal is already processed and given by a function $g\colon \Gamma \times [0,\infty) \to (0,\infty)$. We assume that $g$ is continuous, $\nabla g \in L^2(\Gamma)$, 
and that $g$ satisfies $\gbarmin\leq g\leq \gbarmax$ for positive constants $0<\gbarmin<\gbarmax<1$.
In the following we denote
\begin{equation*}
 \gmax(t):=\max_{\Gamma} g(\cdot,t) \qquad \textrm{ and } \qquad \gmin(t):=\min_{\Gamma} g(\cdot,t)\,.
\end{equation*}

The free boundary problem for the concentration $u \colon \Gamma \times [0,\infty)$ that was derived in \cite{NiRV20,LNRV21} reads as follows:
\begin{align}
  \partial_t u -\Delta u &= -(1-g)\xi + w g \,,& \textrm{ on } \Gamma\times (0,\infty)\,,\label{eq:udk-1}\\
  u \geq 0\,&,\quad u\xi = u\,,\quad 0\leq \xi\leq 1\,& \textrm{ a.e.~on } \Gamma\times (0,\infty)\,, \label{eq:udk-2}\\ 
  -D \Delta w &= 0 & \textrm{ in } \Omega\times (0,\infty) \,,\label{eq:udk-3a}\\
  -D \frac{\partial w}{\partial n} &= - (1-g) \xi + w g &  \textrm{ on } \Gamma\times (0,\infty)\,,
  \label{eq:udk-3b}\\
  u(\cdot,0) &= u_0& \textrm{ on } \Gamma\,. \label{eq:udk-4} 
\end{align}
Here $D>0$ is the ratio of the diffusion constants in the cell and the membrane respectively and $u_0 \in C^0(\Gamma)\cap H^1(\Gamma)$, $u_0 \geq 0$, is the initial datum for $u$. 
The function $w$ is the concentration of the inactive protein in the interior of the cell, whereas $\xi$ is a reduced Michaelis-Menton law and can be viewed as the Langrange multiplier associated to the constraint that $u$ is nonnegative. 

As $D$ is typically large we also consider the following simpler system that arises from \eqref{eq:udk-1}-\eqref{eq:udk-4} in the limit $D \to \infty$:
\begin{align}
  \partial_t u -\Delta u &= -(1-g)\xi + \alpha g & \textrm{ on } \Gamma\times (0,\infty)\,, \label{eq:uk-1}\\
  u \geq 0\,&,\quad u\xi = u\,,\quad 0\leq \xi \leq 1& \textrm{ a.e~on } \Gamma\times (0,\infty) \label{eq:uk-2}\\ 
  \alpha(t) &=\frac{\int_{\{u(\cdot,t)>0\}}(1-g)\,d\sigma}{\int_{\{u(\cdot,t)>0\}}g\,d\sigma}\,,
 & \textrm{ for a.e.~}t>0\,,\label{eq:uk-3}\\
  u(\cdot,0) &= u_0& \textrm{ on } \Gamma\,. \label{eq:uk-4} 
\end{align}
Here $\alpha$ is the analogue of $w$ in \eqref{eq:udk-1}-\eqref{eq:udk-4}, but is constant in space and just time-dependent.
Mathematically speaking, the function $\alpha$ in \eqref{eq:uk-1}-\eqref{eq:uk-4} can be seen as a Lagrange multiplier associated to mass conservation. 

For the solutions of both systems we have mass conservation, that is 
\begin{equation}\label{eq:masscons}
 \int_{\Gamma} u(\cdot,t)\,d\sigma = \int_{\Gamma} u_0\,d\sigma =:m \qquad 
 \textrm{ for all } t>0\,.
\end{equation}

In \cite{NiRV20,LNRV21} we proved global existence and uniqueness of solutions to \eqref{eq:udk-1}-\eqref{eq:udk-4}, respectively \eqref{eq:uk-1}-\eqref{eq:uk-4}, as well as existence, uniqueness and global stability of steady states. Some further results on regularity aspects of the solutions are contained in \cite{LNRV23,LNRV24}.

One appealing feature of the free boundary models above is that they admit a sharp and analytically tractable notion of polarization: a state is polarized precisely when the support of the active protein concentration and its complement both have positive measure. This characterization has made it possible, for instance, to derive a threshold criterion for polarization in terms of the total protein mass. 

\medskip
A second interesting feature of our model is the {\em precise localization} of the active protein at the signal maximum.
In \cite{NiRV20} we first established a corresponding result for steady states in the limit of vanishing mass $m$ and in \cite{FNV26} we described this concentration process in more detail. 

Our goal in the present paper is to study this property in the time dependent setting. 
The analysis of solutions to parabolic equations near points with distinguished behavior by studying suitable blow-up limits has been extensively developed since the seminal
work by Giga and Kohn \cite{GK85,GK87}. 
Here, we first show in Section \ref{S.zeromass} that in the limit of vanishing mass solutions to 
\eqref{eq:udk-1}-\eqref{eq:udk-4} localize instantaneously (see Theorems \ref{P.localization} and \ref{T.localization} for the precise statements) and follow the evolution of the maximizer for time-dependent signals. 
We prove these results for the case $D<\infty$. We note that without stating them explicitly, the same results hold true for the solutions of 
\eqref{eq:uk-1}-\eqref{eq:uk-4}.  In fact, the arguments are even simpler since we do not need
the compensated compactness result, Lemma  \ref{L.wlvklimit}, in that case. 

In Section \ref{S.rescaled} we then investigate this localization process in more detail on the appropriate slow time scale $t=O(m)$. 
We do this first for the case $D=\infty$ in Section \ref{Ss.infiniteD}. More precisely, we consider a sequence of rescaled solutions 
$u_m(x,t):=\frac{1}{m} u(x,mt)$ of \eqref{eq:uk-1}-\eqref{eq:uk-4} and consider the limit $m \to 0$. This sequence of solutions converges to solutions of the analogue of \eqref{eq:uk-1}-\eqref{eq:uk-4} but without the Laplace operator (see 
\eqref{4eq:UEps0DInfty1}-\eqref{4eq:UEps0DInfty3}). For this system we then show that in the limit $t\to \infty$ the support of solutions is monotonically decreasing and concentrates in the set where $g$ attains its maximum. Our main result is 
a characterization of the corresponding limit measure in Theorem \ref{thm:Conc2}. In Section \ref{Ss.finiteD} we establish the corresponding result for $D<\infty$ (see \eqref{eq:65a}-\eqref{eq:66b}). It turns our that the proofs are similar to the case $D=\infty$ and we only sketch  the arguments. 

Finally, in Section \ref{S.examples} we compute the limit measure for some particular examples of $g$. In particular, we study the generic case that $g$ has a finite number of isolated non-degenerate maxima, but we also construct an example where the mass of the active protein oscillates between two maxima in the limit $t \to \infty$.

An interesting feature of the asymptotic limit of the model  \eqref{4eq:UEps0DInfty1}-\eqref{4eq:UEps0DInfty3} 
is that it has the so-called adaptation property, a property that has been extensively studied in systems biology and mathematical biology. 
It states  that the response of a biological system to a signal is independent of the intensity of the signal but rather depends only on the presence of a gradient of concentration (in time or space). 
Notice that this is exactly what happens in the case of the
model \eqref{4eq:UEps0DInfty1}-\eqref{4eq:UEps0DInfty3}.  
Indeed, in this case, the response of the system to any chemical signal
with a non-trivial gradient of concentration is the onset of a peak of
active molecules, strongly concentrated near the maximum of the external
signal. This response is independent of the concentration of the chemical
signal or the strength of the gradient as long as this gradient of
concentration is not too small.

\section{Localization in the zero mass limit}\label{S.zeromass}

In this section we consider a sequence of initial data $(u_{0,k})_k$ that satisfies 
\begin{equation}
  m_k:=\int_\Gamma u_{0,k}\,d\sigma \to 0\qquad \textrm{ as } k \to \infty\,
  \label{eq:uk-5}
\end{equation}
and study the behavior of solutions $(u_k,\xi_k,w_k)_k$ of \eqref{eq:udk-1}-\eqref{eq:udk-4} with initial data $u_{0,k}$ satisfying \eqref{eq:uk-5}.
For notational simplicity we set $D=1$ and denote
 $\alpha_*, \alpha^* \colon [0,\infty) \to (0,\infty)$ by  
\begin{equation}
   \alpha_*(t):= \frac{1-\gmaxt(t)}{\gmaxt(t)} \qquad \mbox{ and } \qquad \alpha^*(t) := \frac{1-\gmin(t)}{\gmin(t)}\,.
  \label{eq:DefAlphaStar}
\end{equation}

\begin{remark}\label{rem:Lh}
For the following it will be convenient to characterize $w$ satisfying \eqref{eq:udk-3a}, \eqref{eq:udk-3b} as
\begin{equation}
 w=L_g\big( (1-g)\xi\big)\,,
 \label{eq:Lhw}
\end{equation}
where 
$L_h:L^2(\Gamma)\to H^1(\Omega)\subset L^2(\Gamma)$, $h\in L^\infty(\Gamma)$, $h\geq 0$, $|\{h>0\}|>0$, is the family of linear operators introduced in \cite{LNRV21}*{p.~1226}, i.e.~$z=L_h(v)$ satisfies
\begin{equation}
	0 =\Delta z \textrm{ in } \Omega, \qquad
	\frac{\partial z}{\partial n}  + hz = v \textrm{ on }\Gamma\,.
  \label{eq:DefLh}%
\end{equation}
By \cite{LNRV21}*{Lemma 2.6} the map $L_h:L^2(\Gamma)\to H^1(\Omega)$ is continuous, self-adjoint and $L_h(h)=1$. 
Using the continuity of $H^1(\Omega)\embeds L^4(\Gamma)$, the continuity of the Neumann-to-Dirichlet map $L^4(\Gamma)\to W^{1,4}(\Gamma)$, see \cite{NiRV20}*{Appendix A}, and the continuous embedding $W^{1,4}(\Gamma)\embeds C^{0,\frac{1}{2}}(\Gamma)$ we obtain 
\begin{align}
  L_h: L^4(\Gamma)\to W^{1,4}(\Gamma)\quad
  \textrm{ and }\quad L_h: L^4(\Gamma)\to C^{0,\frac{1}{2}}(\Gamma)
  \textrm{ are continuous.}
  \label{eq:LhCont}
\end{align}
In addition $h \mapsto L_h$ is monotone decreasing. 
This means for any $h_1,h_2\in L^\infty(\Gamma)$ with $0\leq h_1\leq h_2$ we have
\begin{equation}\label{Lmonotoneh}
  L_{h_1}(s)\geq L_{h_2}(s) \quad\textrm{ for all }s\in L^2(\Gamma),\, s\geq 0\,.
\end{equation}
We also note that it holds 
\begin{equation}\label{Lmonotones}
 L_h(s_1) \geq L_h(s_2) \qquad \mbox{ if } s_1 \geq s_2 \; \mbox{ a.e.}.
\end{equation}

Furthermore it has been shown in \cite[Proposition 2.7]{LNRV21}, that for a solution $(u,\xi,w)$ of \eqref{eq:udk-1}-\eqref{eq:udk-4}
\begin{equation}\label{Lformula}
 L_g\big( (1-g)\xi)= L_{g\chi} \big( (1-g) \chi\big)\,,
\end{equation}
where $\chi$ is the characteristic function of the set $\{x \in \Gamma, t>0 \,|\, u(x,t)>0\}$. 
\end{remark}

The properties of $L$ imply the following a priori bounds on $w$.
 
\begin{lemma}\label{L.aprioriwk}
For almost all $t>0$ we have
\begin{equation*}
 \alpha_*(t) \leq w_k(\cdot,t) \leq \alpha^*(t) \qquad \textrm{ in } \Gamma\,.
\end{equation*}
\end{lemma}
\begin{proof}
We use \eqref{Lformula} with  the characteristic function $\chi_k$ of the set $\{x \in \Gamma, t>0 \,|\, u_k(x,t)>0\}$. 
Together with \eqref{Lmonotoneh}, \eqref{Lmonotones} and $L_h(h)=1$ it follows that 
\begin{align*}
 w_k&= L_{g\chi_k} \big( (1-g)\chi_k\big)  \geq L_{\gmax \chi_k} \big( (1-\gmax) \chi_k\big)\\
 &= \frac{1}{\gmax} L_{\gmax \chi_k} (\gmax \chi_k) - L_{\gmax \chi_k}(\gmax \chi_k)= \frac{1}{\gmax} -1 = \alpha_*\,.
\end{align*}
The upper bound follows analogously.
\end{proof}

We first observe that $u_k$ vanishes in the limit $k\to\infty$.
\begin{lemma} \label{L.uktozero}
 We have $u_k(\cdot,t) \to 0$ in $L^1(\Gamma)$ for all $t>0$ and $u_k \weakto 0$ in $W^{2,1}_2(\Gamma\times (\delta,T))$ for any $0<\delta<T$.
\end{lemma}
\begin{proof}
 We first notice that since $u_k$ is nonnegative we obtain from mass conservation and $m_k \to 0$ that $u_k(\cdot,t) \to 0$ in $L^1(\Gamma)$ for all $t \geq 0$. By standard parabolic regularity, since the right-hand side of \eqref{eq:udk-1} is due to Lemma \ref{L.aprioriwk} uniformly bounded in $L^{\infty}((0,\infty), L^2(\Gamma))$, it follows that $u_k$ is uniformly bounded in $W^{2,1}_2(\Gamma \times (\delta,T))$ for any $0<\delta<T$. Thus there exists a weakly convergent subsequence in this space. Since $u_k(\cdot,t)$ converges to zero in $L^1(\Gamma)$, the limit is zero and the whole sequence converges.
\end{proof}

\bigskip
\begin{proposition}\label{P.wkconv}
It holds
\begin{align}
  w_k &\to \alpha_* \quad\textrm{ in }L^p(\Gamma_T)\textrm{ for any }T>0\,,p \in [1,\infty)\,.\label{eq:WkConv1}
\end{align}
\end{proposition}
\begin{proof}
We fix $T>0$.
By testing \eqref{eq:udk-3a}-\eqref{eq:udk-3b} with $w_k$ and using the lower bound on $g$ we easily find that 
\begin{equation}\label{eq:nablawkbound}
  \sup_{t>0}\| w_k(\cdot,t)\|_{H^1(\Omega)} \leq C(g)\,.
\end{equation}
For a subsequence we then have weak convergence in $L^2(0,T;H^1(\Omega))$ to some limit $\overline{w}$.
Due to Lemma \ref{L.aprioriwk}, the continuity of the trace operator, and an identification argument we deduce that in addition 
\begin{equation} \label{wkconv}
 w_k \weakstarto \overline{w} \qquad \textrm{ in } L^\infty(\Gamma_T)
\end{equation}
holds for the traces of $w_k,\overline{w}$.
We are going to show that $\overline{w}=\alpha_*$, hence it suffices to consider the chosen subsequence.

\begin{steps}
\item For almost all $t>0$
\begin{equation*}
 \overline{w}(\cdot,t) = \overline{w}(t)\quad\textrm{ is constant in }\Omega\,.
\end{equation*}
Since we can write the right hand side of \eqref{eq:udk-3b} as $\partial_t u_k - \Delta u_k$, we see from the weak formulation of \eqref{eq:udk-3a}-\eqref{eq:udk-3b} and Lemma \ref{L.uktozero} that $\overline{w}$ is harmonic inside $\Omega$ with zero Neumann boundary data.
Hence the claim follows. 

\item We have 
\begin{equation}\label{limsupwk}
 \overline{w}(t)  \leq \alpha_*(t) \qquad \textrm{ for a.a. } t>0\,.
\end{equation}

To prove this claim we first multiply equation \eqref{eq:udk-1} with a smooth, nonnegative test function $\varphi$ with compact support in $\Gamma\times(0,T)$ to obtain
\begin{align*}
 0&= \int_0^{\infty} \int_{\Gamma} u_k \big( \Delta \varphi + \partial_t\varphi\big) - \big( ( 1-g)\xi_k - g w_k)\big)\varphi \,d\sigma\,dt\\
 &\geq \int_0^{\infty} \int_{\Gamma} u_k \big( \Delta \varphi + \partial_t\varphi\big) - \big((1-g) - gw_k \big) \varphi\,d\sigma\,dt \,.
\end{align*}
The first term on the right hand side converges to zero due to Lemma \ref{L.uktozero}.
Hence we find in the limit $k \to \infty$ that 
\begin{equation}
  0 \leq \int_0^T \int_{\Gamma} \big(1-(1+\overline{w})g\big)\varphi\,d\sigma\,dt\,.
  \label{eq:22a}
\end{equation}
Next fix an arbitrary Lebesgue point $t_0>0$ of $\overline{w}-\alpha_*$.
For any given $\eps>0$ choose $x_0\in\Gamma$ with $g(x_0,t_0)>\gmaxt(t_0)-\eps$. 
By continuity of $g,\gmaxt$ there exists $\delta=\delta(\eps)$ with $0<\delta<\min\{t_0,T-t_0\}$ and
\begin{equation}
  g\geq \gmaxt-\eps \quad\textrm{ in }B(x_0,\delta)\times (t_0-\delta,t_0+\delta)\,.
  \label{eq:EstGmaxEps}
\end{equation}
Now consider any nonnegative $\eta\in C^\infty_c((-1,1))$, $\int_\R \eta=1$, let $\eta_\eps=\frac{1}{\delta}\eta\big(\frac{\cdot-t_0}{\delta}\big)$ and choose an arbitrary nontrivial, nonnegative test function $\psi_\eps\in C^\infty(\Gamma)$ with $\spt \psi_\eps\subset B(x_0,\delta)$.

Setting $\varphi_\eps(x,t)=\eta_\eps(t)\psi_\eps(x)$ and using $1 = (1+\alpha_*) \gmaxt$ we obtain from \eqref{eq:22a} that
\begin{align}
  0 &\leq  \int_0^\infty\int_\Gamma \big(1-(1+\overline{w})(\gmaxt-\eps)\big)\varphi_\eps\,d\sigma \,dt\nonumber\\
  &= \int_\Gamma\psi_\eps\,d\sigma \int_0^\infty\eta_\eps(t) \big((1+\overline{w})\eps +(\alpha_*-\overline{w})\gmaxt\big)\,dt\,,
\end{align}
which yields after dividing by $\int_\Gamma\psi_\eps$, by letting $\eps\to 0$, $\delta=\delta(\eps)\to 0$, and by using that $t_0$ is a Lebesgue point of $\overline{w}-\alpha_*$ and that $\gmaxt$ is continuous and positive
\begin{equation*}
  \overline{w}(t_0)-\alpha_*(t_0) \leq 0\,.
\end{equation*}
Since almost all $t_0\in (0,T)$ are Lebesgue points of $\overline{w}-\alpha_*$, we conclude that \eqref{limsupwk} holds.

\item We have
\begin{equation}\label{liminfwk}
 \overline{w}(t) = \alpha_*(t) \qquad \textrm{ for a.a. } t>0\,.
\end{equation}
Indeed, the inequality $\overline{w} \geq \alpha_*$ follows directly from the lower bound in Lemma \ref{L.aprioriwk}.
By the previous step \eqref{liminfwk} holds.

\item Finally, since $w_k\geq \alpha_*$ for all $k\in\N$ by Lemma \ref{L.aprioriwk} we deduce from the weak* convergence in $L^\infty$ that for any $T>0$
\begin{equation*}
  0 = \lim_{k\to\infty}\int_0^T\int_\Gamma (w_k -\alpha_*)\,d\sigma\,dt 
  =\lim_{k\to\infty} \| w_k -\alpha_*\|_{L^1(\Gamma_T)}\,, 
\end{equation*}
which gives the $L^1_{\loc}$-convergence of $w_k$ to $\alpha_*$.
Then, since $(w_k)$ is also uniformly bounded due to Lemma \ref{L.aprioriwk} and since the limit is a priori given we also get convergence in $L^p(\Gamma_T)$ for any $p \in [1,\infty)$ and the whole sequence $k\to\infty$.

\end{steps}
\end{proof}

The next result shows that in a mild sense the support of $u_k(\cdot,t)$ has to concentrate on the sets where $g(\cdot,t)$ takes its maximum.
\begin{proposition}[Localization]\label{P.localization}
With $k\to\infty$ we have
\begin{equation}
  \Big(t\mapsto\fint\limits_{\{u_k(\cdot,t)>0\}} \big(\gmaxt(t)-g(\cdot,t)\big)\,d\sigma\,dt\Big)\,\to\, 0 \quad\textnormal{ in }L^1_{\loc}([0,\infty))\,.
  \label{eq:ConvAverage}
\end{equation}
\end{proposition}

\begin{proof}
Since $u_k\in W^{2p,p}_2(\Gamma\times\R^+)$ for any $1\leq p<\infty$ we obtain from \eqref{eq:udk-1}-\eqref{eq:udk-4} for almost all $t>0$ that
\begin{align*}
  0 &= \int_{\{u_k(\cdot,t)>0\}} \partial_t u_k(\cdot,t) -\Delta u_k(\cdot,t)\,d\sigma\\
  &= -\int_{\{u_k(\cdot,t)>0\}} \big(1-g(\cdot,t)\big)\,d\sigma + \int_{\{u_k(\cdot,t)>0\}}w_k(\cdot,t)g(\cdot,t)\,d\sigma\,.
\end{align*}
Using $1=(1+\alpha_*)\gmaxt$ we deduce for any $T>0$ that
\begin{align}
  0&=\int_{0}^{T}\fint_{\{u_k(\cdot,t)>0\}} (1+\alpha_*(t))\gmaxt(t)- (1+w_k(\cdot,t))g(\cdot,t)\,d\sigma\,dt\nonumber\\
  &= \int_{0}^{T}\fint\limits_{\{u_k(\cdot,t)>0\}} (1+\alpha_*(t))\big(\gmaxt(t)-g(\cdot,t)\big)\,d\sigma\,dt 
  - \int_{0}^{T}\fint\limits_{\{u_k(\cdot,t)>0\}}\big(w_k(\cdot,t)-\alpha_*(t)\big)g(\cdot,t)\,d\sigma\,dt\nonumber\\
  &\geq \int_{0}^{T}\fint\limits_{\{u_k(\cdot,t)>0\}} \big(\gmaxt(t)-g(\cdot,t)\big)\,d\sigma\,dt 
  - \int_{0}^{T}\gmax(t)\fint\limits_{\{u_k(\cdot,t)>0\}}\big(w_k(\cdot,t)-\alpha_*(t)\big)\,d\sigma\,dt\,,
  \label{eq:EstWk}
\end{align}
where we have used that $w_k\geq\alpha_*$ by Lemma \ref{L.aprioriwk}.

We denote 
\begin{equation*}
  \delta_k(x,t) := w_k(x,t)-\alpha_*(t) \geq 0\,.
\end{equation*}
We know from Proposition \ref{P.wkconv} that $\delta_k$ is uniformly bounded in $L^{\infty}(\Gamma_T)$ and converges to zero in $L^p(\Gamma_T)$ for any $p \in [1,\infty)$. 
By \eqref{eq:nablawkbound} and the embedding \eqref{eq:LhCont} we have that $\delta_k(\cdot,t) \in C^{0,\frac 1 2 } (\Gamma)$ for almost all $t\in (0,T)$ with bounds that are uniform in $t$.

We fix $\eps>0$ und define
\begin{equation*}
  I_k^{\eps}:=\{ t \in [0,T]\,|\, \max_{ x \in \Gamma} \big(\delta_k(x,t)\big) \geq \eps \}\,.
\end{equation*}
Then let $x^0_k$ be such that 
\begin{equation*}
  \delta_k\big(x^0_k(t),t)\big) = \max_{x \in \Gamma}\big( \delta_k(x,t)\big)\,.
\end{equation*}
Due to the uniform H\"older regularity of $\delta_k(\cdot,t)$ we find that there is $C>0$ such that $\delta_k(x,t) \geq \frac{\eps}{2}$ for $x \in \Gamma$ with $|x-x^0_k| \leq C\eps^2$.
With the bound
\begin{align*}
  \int_0^T \int_{\Gamma} \delta_k(x,t)\,d\sigma\,dt \geq \int_{I^{\eps}_k} \int_{B_{C\eps^2}(x^0_k)} \delta_k(x,t)\,d\sigma\,dt \geq C \eps^{5} |I_k^{\eps}|
\end{align*}
we find that for fixed $\eps>0$ it holds $|I^{\eps}_k| \to 0$ as $k \to \infty$. 
Then
\begin{align*}
  &\int_{0}^{T}\gmax(t)\fint\limits_{\{u_k(\cdot,t)>0\}}\delta_k(\cdot,t)\,d\sigma\,dt\\
  &\qquad = \int_{I^{\eps}_k} \gmax(t)
  \fint\limits_{\{u_k(\cdot,t)>0\}} \delta_k(\cdot,t)\,d\sigma \,dt + \int_{[0,T]\backslash I^{\eps}_k} \gmax(t)
  \fint\limits_{\{u_k(\cdot,t)>0\}} \delta_k \,d\sigma\,dt\leq C \big( |I^{\eps}_k| + \eps\big)\,.
\end{align*}
Letting first $k \to \infty$ and then $\eps \to 0$ shows that
\begin{align*}
  \lim_{k\to\infty}\int_{0}^{T}\gmax(t)\fint\limits_{\{u_k(\cdot,t)>0\}}\delta_k(\cdot,t)\,d\sigma\,dt=0\,.
\end{align*}
Together with \eqref{eq:EstWk} this proves the claim.

\end{proof}

\bigskip
For a sharper result on the concentration property we use an additional assumption on the initial data, that is
\begin{equation}
  \int_\Gamma \frac{1}{m_k}(u_{0,k})^2 \,d\sigma \to 0 \quad\textrm{ as }k\to\infty\,.
  \label{eq:Ass2U0}
\end{equation}
Note that this is a rather mild assumption, which is in particular satisfied for initial data of the form $u_{0,k} = \frac{1}{k}u^0$ for some fixed $u^0\in L^2(\Gamma)$.

\bigskip
\begin{theorem}[Localization II]\label{T.localization}
Assume \eqref{eq:Ass2U0}.
Then the functions $v_k=\frac{1}{m_k}u_k$ converge for suitable subsequences weakly* in $L^\infty_{\omega *}(\R^+;\calM_b(\Gamma))$ to a measure $\mu$, represented as $d\mu (x,t)=d\mu_t(x)\,dt$ with $\mu_t(\Gamma)=1$.

Any such measure $\mu$ is concentrated on the set where $g(\cdot,t)$ takes its maximum, i.e.
\begin{equation}
  \mu_t(\{g(\cdot,t)<\gmaxt(\cdot,t)\})=0\quad\textrm{ for a.e. }t>0\,.
  \label{eq:ConcMu}
\end{equation}
\end{theorem}

We first prepare the proof by a convergence result for the product $w_kv_k$.
\begin{lemma}\label{L.wlvklimit}
Suppose that the assumptions and definitions of Theorem \ref{T.localization} hold.
Then
\begin{equation}\label{eq:wkvklimit}
 \int_0^T \int_{\Gamma} (w_k-\alpha_*) g v_k\,d\sigma \,dt \to 0 \qquad \textrm{ as } k \to \infty\,.
\end{equation}
\end{lemma}
\begin{proof}
We consider $\eps>0$, $k\in\N$ and $\delta_k,I^{\eps}_k$ as in the proof of Proposition \ref{P.localization}
Recalling that $\delta_k$ is uniformly bounded we deduce
\begin{align*}
  \int_0^T \int_{\Gamma} \delta_k g v_k\,d\sigma\,dt
  &= \int_{I^{\eps}_k} \int_{\Gamma} \underbrace{\delta_k g}_{\leq C} v_k\,d\sigma\,dt + \int_{[0,T]\backslash I^{\eps}_k} \int_{\Gamma} \underbrace{\delta_k g}_{\leq C\eps} v_k\,d\sigma\,dt\leq C \big( |I^{\eps}_k| + \eps\big)\,.
\end{align*}
Since $|I^{\eps}_k| \to 0$ with $k\to\infty$,
letting first $k \to \infty$ and then $\eps \to 0$ proves the claim.

\end{proof}

\begin{proof}[Proof of Theorem \ref{T.localization}]
By definition it holds $v_k\geq 0$ and
\begin{equation*}
  \int_\Gamma v_k(\cdot,t)\,d\sigma = 1\quad\textrm{ for all }t>0\,.
\end{equation*}
In particular, $(v_k)_k$ is uniformly bounded in $L^\infty_{\omega *}(\R^+;\calM_b(\Gamma))=L^1(\R^+;C^0(\Gamma))^*$, i.e. the space of weakly* measurable, uniformly bounded functions on $\R^+$ with values in the finite Radon measures on $\Gamma$.

Therefore, there exists a subsequence $k\to\infty$ (not relabeled) and $\mu\in L^\infty_{\omega *}(\R^+;\calM_b(\Gamma))$ such that 
\begin{equation}
  v_k \weakstarto \mu \quad\textrm{ in } L^1(\R^+;C^0(\Gamma))^*\,.
  \label{eq:ConvVK}
\end{equation}
In particular we have a family $(\mu_t)_{t>0}$ of finite Radon measures on $\Gamma$ such that $d\mu(x,t)=d\mu_t(x)\,dt$ and $\mu_t(\Gamma)=1$ for all $t>0$.

We next test equation \eqref{eq:udk-1} with $u_k$ and deduce that
\begin{align*}
  0 &= \int_\Gamma \big(\partial_t u_k -\Delta u_k +(1-g)\xi_k -w_kg\big)u_k\,d\sigma\\
  &= \frac{d}{dt}\int_\Gamma \frac{1}{2}u_k^2\,d\sigma + \int_\Gamma |\nabla u_k|^2\,d\sigma
  +\int_\Gamma (1-g)u_k -w_kg u_k\,d\sigma\,.
\end{align*}
For any $T>0$ this yields
\begin{align*}
  &\int_0^T \int_\Gamma (1-g)v_k -w_kgv_k \,d\sigma\,dt\leq \int_\Gamma \frac{1}{2m_k}(u_{0,k})^2\,d\sigma\,\to 0 \quad\textrm{ as }k\to\infty\,,
\end{align*}
where we have used assumption \eqref{eq:Ass2U0}.

Next by Proposition \ref{P.wkconv} and \eqref{eq:wkvklimit} we can pass to the limit $k\to\infty$ in this inequality and obtain
\begin{align*}
  &\int_0^T \int_\Gamma \big((1-g)  -\alpha_*g\big)\,d\mu_t\,dt\leq 0\,,
\end{align*}
hence by $(1+\alpha_*)\gmaxt=1$
\begin{align*}
  &\int_0^T \int_\Gamma 1 \,d\mu_t\,dt 
  \leq \int_0^T  \int_\Gamma (1+\alpha_*)g \,d\mu_t\,dt 
  = \int_0^T  \int_\Gamma \frac{g}{\gmaxt}\,d\mu_t\,dt
  \leq \int_0^T \int_\Gamma 1 \,d\mu_t\,dt\,.
\end{align*}
This yields for almost all $t>0$ that $\fint_\Gamma \frac{g}{\gmaxt}\,d\mu_t=1$ and hence that $\mu_t(\{g(\cdot,t)<\gmaxt\})=0$.
\end{proof}

\section{Rescaled systems for small times}\label{S.rescaled}
We have seen that in the original time scale solutions instantaneously get supported in the maximizer of $g$.
We now consider a slow time-scale and a rescaling that describes this concentration process.
In this chapter we first consider in Section \ref{Ss.infiniteD} the case $D=\infty$.

\subsection{The rescaled system in the infinite cytosolic diffusion limit}\label{Ss.infiniteD}
Consider for $m>0$ a solution $(u,\xi,\alpha)$ of \eqref{eq:uk-1}-\eqref{eq:uk-3} with $\int_\Gamma u(\cdot,0)\,d\sigma=\int_\Gamma u(\cdot,t)\,d\sigma=m$. 

In order to analyze the concentration process we assume that $g=g(x)$ is independent of time. As before, we denote
\[
\alpha_*:= \frac{1-\gmax}{\gmax} \qquad \mbox{ and } \qquad \alpha^*:= \frac{1-\gmin}{\gmin}\,.
\]
Now we consider the rescaled functions $u_m,\xi_m:\Gamma\times (0,\infty)\to\R$, $\alpha_m:(0,\infty)\to\R$, $u_{0,m}:\Gamma\to\R$ given by
\begin{equation*}
  u_m(x,t)=\frac{1}{m}u(x,mt)\,,\quad \xi_m(x,t)=\xi(x,mt)\,,\quad
  \alpha_m(t)=\alpha(mt)\,,\quad u_{0,m}(x)=\frac{1}{m}u(x,0)\,.
\end{equation*}
Then $(u_m,\xi_m,\alpha_m)_m$ is the unique solution of the system
\begin{align}
  \partial_t u_m & =m \Delta u_m-(1-g)\xi_m + \alpha_mg & \textrm{ on } \Gamma\times (0,\infty)\,\label{eq:RescXi-1}\\
    u_m &\geq 0\,\quad 
  0\leq \xi_m\leq 1\,,\quad 
  u_m\xi_m = u_m & \textrm{ a.e~on } \Gamma\times (0,\infty)\,, \label{eq:RescXi-4}
  \\
   \alpha_m(t) &= \frac{\int_{\{u_m(\cdot,t>0)\}} (1-g)\,d\sigma}{\int_{\{u_m(\cdot,t>0)\}} g\,d\sigma}
  & \textrm{ for a.e.~}t>0\,, \label{eq:RescXi-2}\\
  u_m(\cdot,0) &= u_{0,m}& \textrm{ on } \Gamma\,, \label{eq:RescXi-5}
\end{align}
and satisfies 
\begin{equation*}
  \int_\Gamma u_m(\cdot,t)\,d\sigma=1 \quad\textrm{ for a.e.~}t>0\,.
\end{equation*}

In the following we consider a family of solution $(u_m,\xi_m,\alpha_m)_{m>0}$ of \eqref{eq:RescXi-1}-\eqref{eq:RescXi-5} for a given family of nonnegative initial data $(u_{0,m})_m$ with $\int_\Gamma u_{0,m}\,d\sigma=1$ for all $m>0$ and
\begin{equation}
  u_{0,m} \to u_0\quad\textrm{ in }H^1(\Gamma)\cap C^0(\Gamma)\,. \label{eq:RescAssUm0}
\end{equation}
By mass conservation it holds
\begin{equation}
  \int_\Gamma u_m(\cdot,t)\,d\sigma = 1\quad\textrm{ for a.e.~}t\in(0,T)\,. \label{eq:RescMassCons}
\end{equation}

We first derive a limit system for $m\to 0$.
\begin{lemma}[Limit system] \label{lem:LimSys}
With $m\to 0$ the solutions $(u_m)_{m}$ converge strongly in $L^2_{\loc}(\Gamma\times[0,\infty))$ and weakly in $H^1_{\loc}(\Gamma\times[0,\infty))$ to the unique solution $(u,\xi,\alpha)\in H^1_{\loc}(\Gamma\times[0,\infty))\times L^\infty(\Gamma\times[0,\infty))\times L^\infty(0,\infty)$ of
\begin{align}
  \partial_t u &= -(1-g)\xi + \alpha g\,,& \mbox{ on } \Gamma\times (0,\infty)\,\label{4eq:UEps0DInfty1}\\
  0\leq \xi &\leq 1\,,\quad u\xi=u\,,u\geq 0\,, & \mbox{ a.e. on } \Gamma \times (0,\infty),
\label{4eq:UEps0DInfty}\\
  \alpha(t) &=\frac{\int_{\{u(\cdot,t)>0\}}(1-g)\,d\sigma}{\int_{\{u(\cdot,t)>0\}}g\,d\sigma}\,,& \mbox{ for a.a. } t>0\,,
  \label{4eq:UEps0DInfty2}\\
  u(\cdot,0) &= u_0\,& \mbox{ on } \Gamma\,.\label{4eq:UEps0DInfty3}
\end{align}
\end{lemma}

\begin{proof}
Choose any $T>0$.

We first derive uniform estimates and compactness properties.
We clearly have $0<\alpha_* \leq \alpha \leq \alpha^*$ and deduce by the maximum principle that
\begin{equation}
  0 \leq u_m \leq C(T)\quad\textrm{ in }\Gamma_T\,. \label{eq:RescBdLinfty}
\end{equation}
By testing equation \eqref{eq:RescXi-1} with $-\Delta u_m$ and using that $\Delta u_m=0$ almost everywhere in $\{u_m=0\}$ we deduce that
\begin{align}
  \int_\Gamma \frac{1}{2}|\nabla u_m(\cdot,t)|^2 \,d\sigma + \int_\Gamma m(\Delta u_m)^2\,d\sigma\,dt
  &\leq \int_\Gamma \frac{1}{2}|\nabla u_{0,m}|^2 \,d\sigma+ \int_0^T\int_\Gamma |\nabla u_m|^2 \,d\sigma\,dt + CT \int_{\Gamma} |\nabla g|^2\,d\sigma\,.
  \label{eq:RescGradBd}
\end{align}
The assumption on the initial data \eqref{eq:RescAssUm0} and Gronwalls Lemma yield
\begin{align}
  \|u_m\|_{L^\infty(0,T;H^1(\Gamma))} +\sqrt m\|\Delta u_m\|_{L^2(\Gamma_T)}\leq C(T)\,,
\end{align}
which by \eqref{eq:RescXi-1} in addition shows that 
\begin{equation}
  \|\partial_t u_m\|_{L^2(\Gamma_T)}\leq C(T)\,.
\end{equation}
By the uniform bounds and Rellich's theorem we conclude that there exists a subsequence $m_k\to 0$ ($k\to\infty$), and functions $u\in H^1(\Gamma_T)$, $\xi\in L^\infty(\Gamma_T)$ and $\alpha\in L^\infty(0,T)$ such that
\begin{align}
  u_{m_k}&\to u \quad\textrm{ in }L^2(\Gamma_T)\,,\quad u\geq 0\,,\\
  u_{m_k}&\weakto u \quad\textrm{ in }H^1(\Gamma_T)\,\\
  \xi_{m_k}&\weakstarto \xi \quad\textrm{ in }L^\infty(\Gamma_T)\,\,\quad
  0\leq \xi\leq 1\,,\\
  \alpha_{m_k}&\weakstarto \alpha \quad\textrm{ in }L^\infty(\Gamma_T)\,.
\end{align}

We next derive the limit system.

For any $\varphi\in C^0(\Gamma_T)$
\begin{equation*}
  \int_{\Gamma_T} u(1-\xi)\varphi \,d\sigma \,dt= \lim_{k\to\infty}\int_{\Gamma_T} u_{m_k}(1-\xi_{m_k})\varphi\,d\sigma \,dt = 0\,,
\end{equation*}
hence
\begin{equation}
  u\xi = u \quad\textrm{ almost everywhere in }\Gamma_T\,.
  \label{eq:RescLim-3}
\end{equation}
By mass conservation \eqref{eq:RescMassCons} the assumption on the initial data \eqref{eq:RescAssUm0}, and the strong $L^2(\Gamma_T)$ convergence of $(u_{m_k})_k$ we obtain
\begin{equation}
  \int_\Gamma u(\cdot,t)\,d\sigma = \int_\Gamma u_0\,d\sigma \qquad\textrm{ a.e. }t\in(0,T)\,. \label{eq:RescMassConsLim}
\end{equation}
We now test \eqref{eq:RescXi-1} with $\varphi\in C^0_c(0,T;C^2)$ and deduce
\begin{align*}
  &\int_{\Gamma_T} \varphi\Big(\partial_t u +(1-g)\xi -\alpha g\Big)\,d\sigma \,dt\nonumber\\
  &\qquad=\lim_{k\to\infty} \int_{\Gamma_T} \varphi\Big(\partial_t u_{m_k} +(1-g)\xi_{m_k} -\alpha_{m_k} g\Big)\,d\sigma \,dt-m\int_{\Gamma_T}u_{m_k}\Delta\varphi\,d\sigma \,dt
  =0\,.
\end{align*}
This proves that
\begin{equation}
  \partial_t u = -(1-g)\xi +\alpha g\quad\textrm{ a.e.~in }\Gamma_T\,.
  \label{eq:RescLim-1}
\end{equation}
Using the mass conservation \eqref{eq:RescMassConsLim}, the uniform boundedness of $\partial_t u$, and $\partial_t u=0$ almost everywhere in $\{u(\cdot,t)=0\}$ yields for almost all $t\in (0,T)$
\begin{equation*}
  0 = \int_\Gamma \partial_tu(\cdot,t) \,d\sigma
  = \int_{\{u(\cdot,t)>0\}} \partial_t u(\cdot,t)\,d\sigma
  = \int_{\{u(\cdot,t)>0\}} -(1-g) +\alpha(t) g(\cdot,t)\,d\sigma\,,
\end{equation*}
which shows
\begin{equation*}
  \alpha(t) = \frac{\fint_{\{u(\cdot,t)>0\}}(1-g)\,d\sigma}{\fint_{\{u(\cdot,t)>0\}}g\,d\sigma}\quad\textrm{ for a.e.~}t\in (0,T)\,.
  \label{eq:RescLim-2}
\end{equation*}
We next observe that solutions on $(0,T)$ are unique.
This follows from the $L^1$-contraction property proved in \cite[Theorem 3.1]{LNRV21}.
Reviewing that proof we easily see that for any two solutions $(u_1,\xi_1,\alpha_1)$ and $(u_2, \xi_2,\alpha_2)$ we have
\begin{equation*}
  t\mapsto \int_{\Gamma}  (u_1-u_2)_+(\cdot,t)\,d\sigma \qquad \textrm{ is decreasing on } [0,T]\,.
\end{equation*}
In particular, given $u_0 \in L^2(\Gamma)$ with  $u_0 \geq 0$, there exists at most one solution $(u,\xi,\alpha)$ on $(0,T)$.

By a familiar identification and continuation argument this shows the convergence towards a unique solution of \eqref{4eq:UEps0DInfty1}-\eqref{4eq:UEps0DInfty3}.
\end{proof}

Next let us consider the limit problem.
We will show that the support of $u(\cdot,t)$ is decreasing with time.

We define a set $G$ of well-behaved times as those values $t\in (0,T)$ that are a Lebesgue point of $\alpha,\xi$, i.e.
\begin{equation}
  \alpha(t) = \lim_{h\to 0} \fint_{t-h}^{t+h}\alpha(\tau)\,d\tau\,,\qquad
  \xi(\cdot,t) = \lim_{t\downarrow 0} \fint_{t-h}^{t+h}\xi(\cdot,\tau)\,d\tau\quad\textrm{ in }L^2(\Gamma)\,,
  \label{eq:AssLebPt}
\end{equation}
and that satisfy
\begin{equation}
  0=(1-g)\xi(\cdot,t) -\alpha(t) g \quad\textrm{ almost everywhere in }\{u(\cdot,t)=0\}\,,
  \label{eq:Resc-G2}
\end{equation}
as well as
\begin{equation}
  0=\int_\Gamma\Big((1-g)\xi(\cdot,t) -\alpha(t) g\Big)\,d\sigma
  \label{eq:Resc-G3}
\end{equation}
We note that $(0,T)\setminus G$ is a $\calL^1$-null set.

\begin{proposition}[Monotonicity of support] \label{prop:MonSupp}
Let $(u,\xi,\alpha)$ be a solution of \eqref{4eq:UEps0DInfty1}-\eqref{4eq:UEps0DInfty3}.
Then for any $0\leq t_0<t$, $t_0,t\in G$ it holds
\begin{equation}
  u(\cdot,t)=0 \quad\textrm{ a.e.~in }\{u(\cdot,t_0)=0\}\,.
  \label{eq:MonSupp}
\end{equation}
\end{proposition}

\begin{proof}
We assume $t_0=0\in G$, i.e.~that the system at time $t=0$ is well behaved.
Without loss of generality we assume $u_0\geq 0$ and $\int_\Gamma u_0=1$, and we define
\begin{equation*}
  \alpha_0 := \frac{\fint_{\{u_0>0\}}(1-g)\,d\sigma}{\fint_{\{u_0>0\}}g\,d\sigma}\quad\textrm{ for a.e.~}t\in (0,T)\,,
\end{equation*}
and
\begin{equation}
  \xi_0 := 
  \begin{cases}
    1\quad&\textrm{ in }\{u_0>0\}\,,\\
    \alpha_0\frac{g}{1-g} &\textrm{ in }\{u_0=0\}\,.
  \end{cases}
  \label{eq:DefXi0}
\end{equation}
By assumption
\begin{equation}
  \alpha_0 = \lim_{h\downarrow 0} \fint_0^h\alpha(\tau)\,d\tau\,,\qquad
  \xi_0 = \lim_{h\downarrow 0} \fint_0^h\xi(\cdot,\tau)\,d\tau\quad\textrm{ in }L^2(\Gamma)\,.
  \label{eq:AssLebPt-0}
\end{equation}

For any $t\in G$ we deduce from \eqref{eq:Resc-G2} and the definition of $\alpha_0,\xi_0$ that
\begin{equation}
  0=(1-g)\big(\xi(\cdot,t)-\xi_0\big) -\big(\alpha(t)-\alpha_0\big) g \quad\textrm{ almost everywhere in }\{u(\cdot,t)+u_0=0\}\,.
\end{equation}
Together with \eqref{eq:Resc-G3} this yields with $\Chi_0:=\Chi_{\{u_0>0\}}$, $\Chi(\cdot,t):=\Chi_{\{u(\cdot,t)>0\}}$ and  $\Chi^*(\cdot,t):=\Chi_{\{u(\cdot,t)+u_0>0\}}$ that
\begin{equation*}
  0=\int_\Gamma\Chi^*(\cdot,t)\Big((1-g)\big(\xi(\cdot,t)-\xi_0\big) -\big(\alpha(t)-\alpha_0\big) g \Big)\,d\sigma \,,
\end{equation*}
hence
\begin{equation}
  \big(\alpha(t)-\alpha_0\big) \int_\Gamma g\Chi^*(\cdot,t) \,d\sigma
  =\int_\Gamma\Chi^*(\cdot,t)(1-g)\big(\xi(\cdot,t)-\xi_0\big)\,d\sigma\,.
  \label{eq:AlphaDiff}
\end{equation}

Moreover, we deduce from \eqref{eq:RescLim-1}, \eqref{eq:DefXi0}, and \eqref{eq:AlphaDiff} that for almost all $t\in (0,T)$
\begin{align*}
  &\int_\Gamma g\Chi^*(\cdot,t)\,d\sigma\frac{d}{dt}\int_\Gamma (1-\Chi_0)u(\cdot,t)\,d\sigma \\
&\qquad = \int_\Gamma g\Chi^*(\cdot,t)\,d\sigma\int_\Gamma (1-\Chi_0)\Chi(\cdot,t)\Big(-(1-g)\xi +\alpha g\Big)\,d\sigma\\
&\qquad = \int_\Gamma g\Chi^*(\cdot,t)\,d\sigma\int_\Gamma (1-\Chi_0)\Chi (\cdot,t)\Big(-(1-g)(\xi-\xi_0) +(\alpha-\alpha_0)g \Big)\,d\sigma\\
  &\qquad = -\int_\Gamma g\Chi^*(\cdot,t)\,d\sigma \int_\Gamma (1-\Chi_0)\Chi (\cdot,t)(1-g)(\xi-\xi_0)\,d\sigma\\
 &\qquad\qquad + \int_\Gamma (1-\Chi_0)\Chi (\cdot,t) g \,d\sigma
 \int_\Gamma\Chi^*(\cdot,t)(1-g)\big(\xi(\cdot,t)-\xi_0\big)\,d\sigma\\
 &\qquad = -\int_\Gamma g\Chi_0\,d\sigma \int_\Gamma (1-\Chi_0)\Chi (\cdot,t)(1-g)(\xi-\xi_0)\,d\sigma\\
&\qquad\qquad + \int_\Gamma (1-\Chi_0)\Chi (\cdot,t)g\,d\sigma
  \int_\Gamma\Chi_0(1-g)\big(\xi(\cdot,t)-\xi_0\big)\,d\sigma \leq 0\,,
\end{align*}
where we have used in the second-last equality that $\Chi^*(\cdot,t)=(1-\Chi_0)\Chi (\cdot,t)+\Chi_0$ and in the final inequality that $\xi-\xi_0\geq 0$ in $\{u(\cdot,t)>0\}$ and $\xi(\cdot,t)-\xi_0\leq 0$ in $\{u_0>0\}$.

This proves the claim.
\end{proof}

\begin{corollary}[Representation formula]
For any solution $(u,\xi,\alpha)$ of \eqref{4eq:UEps0DInfty1}-\eqref{4eq:UEps0DInfty3} we have for almost all $t>0$
\begin{equation}
  u(\cdot,t) = \Big(u_0 + \int_0^t -(1-g)(\cdot,\tau)+\alpha(\tau)g\,d\tau\Big)_+\,.
  \label{eq:RepForm}
\end{equation}
In particular, $u:\Gamma\times[0,\infty)$ has a continuous representative.
\end{corollary}

\begin{proof}
By monotonicity of the support, we deduce for all $t\in G$
\begin{align*}
  u(\cdot,t) = u_0 + \int_0^t -(1-g)(\cdot,\tau)+\alpha(\tau)g\,d\tau\quad \textrm{ in } \{u(\cdot,t)>0\}\,.
\end{align*}
Moreover, since $\xi \le 1$ and $u\in H^1_{\loc}(\Gamma\times[0,\infty))$, it follows that for almost all $t\in G$
\begin{align*}
  0 &=\partial_tu(\cdot,t) \ge u_0 + \int_0^t -(1-g)(\cdot,\tau)+\alpha(\tau)g\dd\tau
  \quad \textrm{ a.e.~in } \{u(\cdot,t)=0\}\,.
\end{align*}

Therefore, \eqref{eq:RepForm} follows.
\end{proof}

We deduce a monotonicity property also for $\alpha$ and $\xi$.
\begin{corollary}[Monotonicity of $\alpha$ and $\xi$] \label{Cor:Mon}
Let $(u,\xi,\alpha)$ be a solution of \eqref{4eq:UEps0DInfty1}-\eqref{4eq:UEps0DInfty3}.
Then for almost all $0\leq t_0<t$ we have
\begin{align*}
  \alpha(t)&\leq \alpha(t_0)\,,\qquad
  \xi(\cdot,t) \leq \xi(\cdot,t_0)\quad\textrm{ a.e.~in }\Gamma\,.
\end{align*}
\end{corollary}

\begin{proof}
For $t_0<t \in G$ as above we deduce by \eqref{4eq:UEps0DInfty1}, \eqref{eq:MonSupp}, by $\xi\leq 1$, and \eqref{4eq:UEps0DInfty2}
\begin{equation*}
  \alpha(t)= 
  \frac{\int_{\{u(\cdot,t_0)>0\}}(1-g)\xi(\cdot,t)\,d\sigma}{\int_{\{u(\cdot,t_0)>0\}}g\,d\sigma}
  \leq \alpha(t_0)\,.
\end{equation*}
Next, by the properties of $\xi$ we have $\xi(\cdot,t_0)=1\geq \xi(\cdot,t)$ in $\{u(\cdot,t_0)>0\}$.
Moreover in $\{u(\cdot,t_0)=0\}\subset\{u(\cdot,t)=0\}$ it holds
\begin{equation*}
  (1-g)\big(\xi(\cdot,t)-\xi(\cdot,t_0)\big)
  = g\big(\alpha(t)-\alpha(t_0)\big)\leq 0\,,
\end{equation*}
and the monotonicity of $\xi$ also follows.
\end{proof}

In the following we always choose the continuous representative of $u$ and the left-continuous decreasing representatives of $\alpha,\xi$ on $(0,\infty)$.

\begin{theorem}[Concentration of support] \label{thm:Conc1}
We have 
\[
 \alpha(t)\searrow \alpha_*
\]
and there exists a Borel set $S\subset\Gamma$ such that
\begin{equation}
  S=\mathsmaller\bigcap\limits_{t>0} S(t) \qquad \mbox{ with } S(t)= \{u(\cdot,t)>0\}\,.
  \label{eq:LimSWInfty-DInfty}
\end{equation}
Then $S\subset \{g=\gmax\}$ holds and for any sequence $t_k\to\infty$ there exists a subsequence (not relabeled) and a Radon measure $\mu$ on $\Gamma$ with support in $\bar S$ such that
\begin{equation}
  \mu = \lim_{k\to\infty} u(\cdot,t_k)\calH^2 \quad\textrm{ in the sense of Radon measures on }\Gamma\,.
  \label{eq:CompactUT}
\end{equation}
\end{theorem}

\begin{proof}
\begin{steps}
\item\label{Thm3.4:step1} The set $S$ is Borel since it is given as countable intersection $\bigcap_{\ell\in\N}\{u(\cdot,\ell)>0\}$ of open sets (here we use the monotonicity of support).
The existence of some $\alpha_\infty$ with $\alpha(t)\searrow \alpha_\infty$ follows by the monotonicity of $\alpha$.
Since $\alpha(t)\geq\alpha_*$ by \eqref{4eq:UEps0DInfty2} we clearly have $\alpha_\infty\geq\alpha_*$.
Similarly, we have $\xi(\cdot,t)\to \xi_\infty$ in every $L^p(\Gamma)$, $1\leq p<\infty$ for some $0\leq \xi_\infty\leq 1$.

Moreover, the compactness property \eqref{eq:CompactUT} follows since the total mass of $u(\cdot,t)$ is constant in $t$.
By the monotonicity of the support of $u(\cdot,t_k)$ it follows that $\spt(\mu)\subset\bar S$ holds.

\item Fix an arbitrary $t_k>t_0$ with $t_k\to\infty$ and $G$ as above.
Then $u(\cdot,t_k)=0$ in $\{u(\cdot,t_0)=0\}$ and in this set
\begin{equation*}
  (1-g)\xi_\infty=(1-g)\lim_{k\to\infty}\xi(\cdot,t_k)
  = g \lim_{k\to\infty}\alpha(t_k)
  = g \alpha_\infty
\end{equation*}
holds.
This yields
\begin{equation}
  (1-g)\xi_\infty=g \alpha_\infty \quad\textrm{ a.e.~in }\mathsmaller\bigcup\limits_{t>0}\{u(\cdot,t)=0\}=\Gamma\setminus S\,.
\label{eq:1001}
\end{equation}

\item We claim that
\begin{equation}
  \alpha_\infty \geq \frac{1-g}{g}\geq \alpha_* \qquad\textrm{ in }S\,,
  \label{eq:TildeAlph}
\end{equation}
holds. 
Otherwise, there exists $x^*\in S$ with $\alpha_\infty < \frac{1-g(x^*)}{g(x^*)}$.
By continuity of $g$ and the monotone convergence $\alpha(t_k)\searrow \alpha_\infty$ there exist $\eps,\delta,T>0$ such that
\begin{equation*}
  \alpha(t) \leq \frac{1-g}{g}-\eps\quad\textrm{ in }B(x^*,\delta)\textrm{ for all }t\geq T\,.
\end{equation*}
This implies that for any $x\in B(x^*,\delta)$ and all $t\geq T$ with $x\in S(t)$ it holds
\begin{equation*}
  \partial_tu(x,t) = -1 + g(x)\big(1+\alpha(t)\big) \leq -\eps \min g\,.
\end{equation*}
By nonnegativity of $u$ we then conclude that $B(x^*,\delta)\cap S(t)=\emptyset$ for all $t$ sufficiently large, a contradiction to $x^*\in S$.
This proves \eqref{eq:TildeAlph}.

\item We claim that
\begin{equation}
  \alpha_\infty = \frac{1-g}{g} \quad\calH^2-\textrm{a.e.~in }S\,.
\label{eq:alphainftyIdent}
\end{equation}
Otherwise there exists $\eps>0$ and $x^*\in S$ with
\begin{equation*}
  \calH^2\big(S\cap B(x^*,r)\big)>0 \quad\textrm{ for all }r>0\,,\quad\textrm{ and }\alpha_\infty - \frac{1-g}{g}(x^*)> \eps\,.
\end{equation*}
By continuity of $g$ and monotone convergence of $\alpha(t)$ as $t \to \infty$ we conclude that for some $\delta>0$ it holds $\alpha - \frac{1-g}{g}(x^*)> \eps$ on $B(x^*,\delta)$, hence
\begin{equation*}
  \partial_tu = -1 + g\big(1+\alpha\big) \geq \eps \min_\Gamma g \quad\textrm{ in }S\cap B(x^*,\delta)\,,
\end{equation*}
which gives a contradiction to $\int_\Gamma u=1$ and proves \eqref{eq:alphainftyIdent}.

\item By $\xi_\infty\leq 1$, \eqref{eq:1001}, \eqref{eq:alphainftyIdent} we then have
\begin{equation*}
  \alpha_\infty \leq \frac{1-g}{g}\quad\textrm{ a.e.~in }\Gamma\,,
\end{equation*}
which, evaluated in $\argmax g$ and recalling $\alpha_\infty\geq\alpha_*$,  together with \eqref{eq:TildeAlph} implies $\alpha_\infty=\alpha_*$.

\item We finally prove that $g=\gmax$ holds in $S$.
Otherwise, there exists $x_* \in S$ such that $g(x_*)<\gmax$.
Together with the continuity of $g$ this yields the existence of $\eps>0$ and $\delta>0$ with
\begin{equation*}
  g(x) \leq \gmax-\eps \quad\textrm{ for all } x \in B(x_*,\delta)\,.
\end{equation*}
For any $\eps_1>0$ there exists $t_1=t_1(\eps_1)>0$ such that
\begin{align*}
  \alpha(t) \leq \alpha_*+\eps_1 \quad\textrm{ for all } t \geq t_*\,
\end{align*}
We then deduce for $\eps_1>0$ sufficiently small that in $B(x_*,\delta)\cap S$
\begin{align*}
  \partial_t u(\cdot,t) &=-1+g(1+\alpha(t)) \\
  &\leq -1+\left(\gmax-\eps\right)\left(1+\alpha_*+\eps_1\right) 
  \leq -\eps(1+\alpha_*)+\gmax\eps_1<0 
\end{align*}
for all $t>t_1$, a contradiction to $x_*\in S$.
\end{steps}
\end{proof}

Our main result is the following characterization of the limit.
\begin{theorem}[Distribution of limit measures]\label{thm:Conc2}
Consider $S$, $\alpha_*$ as in \eqref{eq:LimSWInfty-DInfty} and assume $\calH^2(S)=0$.
Let a sequence $t_k\to\infty$ and a Radon measure $\mu$ on $\Gamma$ with support in $S$ be given such that
\begin{equation}
  \mu = \lim_{k\to\infty} u(\cdot,t_k)\calH^2 \quad\textrm{ in the sense of Radon measures on }\Gamma\,.
  \label{eq:LimMu}
\end{equation}
Then there exists $\omega(k)\to 0$ for $k\to\infty$ such that for any measurable set $A\subset \Gamma$ with $\mu(A)>0$
\begin{equation*}
  \mu(A) = \lim_{k\to\infty}
  \frac{
    \int_A
      \big(g-(\gmax-\omega(k))\big)_+\,d\sigma
  }{
    \int_\Gamma
      \big(g-(\gmax-\omega(k))\big)_+\,d\sigma
  }\,.
\end{equation*}
\end{theorem}

\begin{proof}
Let
\begin{equation*}
  \varphi := (\alpha-\alpha_*)\,\gmax\,,\qquad
  h := \frac{\gmax-g}{\gmax}\,.
\end{equation*}
Then $g = (1-h)\,\gmax$ and $\alpha = \frac{\varphi}{\gmax} + \alpha_*$.

Consequently, we obtain
\begin{align*}
  \partial_t u &= -(h-\varphi+h\varphi)\,\chi_{\{u>0\}} \quad\textrm{ in }\Gamma\times (0,\infty)\,,\\
  \varphi(t) &= \frac{\int_{\{u(\cdot,t)>0\}}h\,d\sigma}{\int_{\{u(\cdot,t)>0\}}(1-h)\,d\sigma}\quad\textrm{ for a.e.~}t>0\,.
\end{align*}
and from \eqref{eq:RepForm} that

\begin{align*}
  u(\cdot,t)
  = \left(u_0 + \int_0^t (\varphi - h - \varphi h)(\cdot,\tau)\dd\tau\right)_{+}\,.
\end{align*}
By mass conservation, we may assume without loss of generality that
\begin{equation*}
  \int_{\{u(\cdot,t)>0\}} u(\cdot,t)\,d\sigma
  =\int_{\Gamma} u(\cdot,t)\,d\sigma= 1 \quad \mbox{ for all } t>0\,.
\end{equation*}
For any measurable subset $A \subset \Gamma$ we deduce
\begin{align}
  \int_{A} u(\cdot,t)\,d\sigma
  &= \frac{
    \int_{A\cap\{u(\cdot,t)>0\}}\left(u_{0} + \int_0^t(\varphi-h-\varphi h)(\cdot,\tau)\dd\tau\right)\,d\sigma 
  }{
    \int_{\{u(\cdot,t)>0\}}\left(u_{0} + \int_0^t(\varphi-h-\varphi h)(\cdot,\tau)\dd\tau\right)\,d\sigma
  }\,. 
  \label{4eq:MuA}
\end{align}
By the definition of $S$ and the assumption $\calH^2(S)=0$ we have 
\begin{equation*}
  \lim_{k\to\infty}\int_{\{u(\cdot,t_k)>0\}} u_{0}\,d\sigma =0\,.
\end{equation*}
If $\mu(A)>0$ we further deduce
\begin{align}
  0<\mu(A)
  &= \lim_{k\to\infty}\int_{A} u(\cdot,t_k)\,d\sigma\nonumber\\
  &= \lim_{k\to\infty}\Biggl[
    \int_{A\cap \{u(\cdot,t_k)>0\}} u_{0}\,d\sigma
    + \int_{A\cap\{u(\cdot,t_k)>0\}} \int_{0}^{t_k}\bigl((1-h)\varphi(\tau)-h\bigr)\dd\tau\,d\sigma
  \Biggr] \nonumber\\
  &= \lim_{k\to\infty}\int_{A\cap \{u(\cdot,t_k)>0\}} \int_{0}^{t_k}\bigl((1-h)\varphi(\tau)-h\bigr)\dd\tau\,d\sigma\,.
  \label{4eq:MuA-2}
\end{align}
Consequently, there exists a sequence $\omega_k\to 0$ as $k\to\infty$ such that
\begin{align*}
  &\int_{A\cap\{u(\cdot,t_k)>0\}} u_{0}\,d\sigma
  + \int_{A\cap \{u(\cdot,t_k)>0\}} \int_{0}^{t_k}\bigl((1-h)\varphi(\tau)-h\bigr)\dd\tau\,d\sigma\\
  &\qquad = (1+\omega_k)\int_{A\cap\{u(\cdot,t_k)>0\}} \int_{0}^{t_k}\bigl((1-h)\varphi(\tau)-h\bigr)\dd\tau\,d\sigma\,.
\end{align*}
Similarly, there exists a sequence $\tilde\omega_k\to 0$ as $k\to\infty$ such that
\begin{align*}
  &\int_{\{u(\cdot,t_k)>0\}} u_{0}\,d\sigma
  + \int_{\{u(\cdot,t_k)>0\}} \int_{0}^{t_k}\bigl((1-h)\varphi(\tau)-h\bigr)\dd\tau\,d\sigma\\
  &\qquad = (1+\tilde\omega_k)\int_{\{u(\cdot,t_k)>0\}} \int_{0}^{t_k}\bigl((1-h)\varphi(\tau)-h\bigr)\dd\tau\,d\sigma\,.
\end{align*}

We then deduce from \eqref{4eq:MuA}
\begin{equation*}
  \mu(A)
  = \lim_{k\to\infty}
  \frac{
    \int_{A\cap\{u(\cdot,t_k)>0\}}
      \left[(1-h)\int_0^{t_k}\varphi(\tau)\dd\tau - t_kh\right]\,d\sigma
  }{
    \int_{\{u(\cdot,t_k)>0\}}
      \left[(1-h)\int_0^{t_k}\varphi(\tau)\dd\tau - t_kh\right]\,d\sigma
  }\,. 
\end{equation*}

Consider
\begin{equation*}
  \tilde{S}_{k}
  :=
  \left\{
    x\in\Gamma :
    \int_{0}^{t_k}\bigl((1-h(x))\varphi(\tau)-h(x)\bigr)\dd\tau > 0
  \right\}
  \subset
  \{ u(\cdot,t_k)>0\}\,. 
\end{equation*}
Then, for $x\in \{u(\cdot,t_k)>0\}\setminus \tilde{S}_{k}$ we have
\begin{equation*}
  0 < u(x,t_k)
  = u_{0}(x) - \left|\int_{0}^{t_k}\bigl((1-h(x))\varphi(\tau)-h(x)\bigr)\dd\tau\right|\,.
\end{equation*}
Hence,
\begin{align*}
  &\Big|\int_{\{u(\cdot,t_k)>0\}\setminus \tilde{S}_{k}}
  \int_{0}^{t_k}\bigl((1-h)\varphi(\tau)-h\bigr)\dd\tau\,d\sigma\Big| \leq   \int_{\{u(\cdot,t_k)>0\}\setminus \tilde{S}_{k}}
  u_0\,d\sigma\,\to\,0\,(k\to\infty)\,.
\end{align*}
Thus, up to negligible terms,
\begin{equation*}
  \int_{\Gamma} u(\cdot,t_k)\dd S
  = \int_{\tilde{S}_{k}}
      \left((1-h)\int_{0}^{t_k}\varphi(\tau)\dd\tau - t_kh\right)\,d\sigma\,,
\end{equation*}
and, using in addition $h\Chi_{\{u(\cdot,t)>0\}}\to 0$ in $L^1(\Gamma)$ for $t\to\infty$,
\begin{equation*}
  \mu(A)
  = \lim_{k\to\infty}
  \frac{
    \int_{\{\,h < \fint_0^{t_k}\varphi(\tau)\dd\tau\,\}\cap A}
      \left(\fint_0^{t_k}\varphi(\tau)\dd\tau - h\right)\,d\sigma
  }{
    \int_{\{\,h < \fint_0^{t_k}\varphi(\tau)\dd\tau\,\}}
      \left(\fint_0^{t_k}\varphi(\tau)\dd\tau - h\right)\,d\sigma
  }\,. 
\end{equation*}

Now $\fint_0^{t_k}\varphi(\tau)\dd\tau \to 0$ as $k\to\infty$, since $\varphi(\tau)\to 0$, and we therefore obtain
\begin{align*}
  \mu(A)
  &= \lim_{k\to\infty}
  \frac{
    \int_{\{\,h < \omega(k)\,\}\cap A}
      \left(\omega(k) - h\right)\,d\sigma
  }{
    \int_{\{\,h < \omega(k)\,\}}
      \left(\omega(k) - h\right)\,d\sigma
  }
  = \lim_{k\to\infty}
  \frac{
    \int_A
      \left(\omega(k) - h\right)_+\,d\sigma
  }{
    \int_\Gamma
      \left(\omega(k) - h\right)_+\,d\sigma
  }\,,
\end{align*}
with $\omega(k)\to 0$ for $k\to\infty$.

\end{proof}

\subsection{The rescaled system for the bulk--surface PDE system}\label{Ss.finiteD}

Consider for $m>0$ a solution $(u,\xi,w)$ of \eqref{eq:udk-1}-\eqref{eq:udk-3b} with $\int_\Gamma u(\cdot,0)\,d\sigma=\int_\Gamma u(\cdot,t)\,d\sigma=m$. 
In the following, for ease of notation, we assume $D=1$.

We again assume that $g=g(x)$ is independent of time and consider the rescaled functions $u_m,\xi_m:\Gamma\times (0,\infty)\to\R$, $w_m:\overline{\Omega}\times(0,\infty)\to\R$, $u_{0,m}:\Gamma\to\R$ given by
\begin{equation*}
  u_m(x,t)=\frac{1}{m}u(x,mt)\,,\quad \xi_m(x,t)=\xi(x,mt)\,,\quad
  w_m(x,t)=w(x,mt)\,,\quad u_{0,m}(x)=\frac{1}{m}u(x,0)\,.
\end{equation*}
Then $(u_m,\xi_m,w_m)$ is the unique solution of the system
\begin{align}
  \partial_t u_m & =m \Delta u_m-(1-g)\xi_m + w_mg & \textrm{ on } \Gamma\times (0,\infty)\label{eq:RescD-1}\\
  w_m(\cdot,t) &=L_{g}\big((1-g)\xi_m(\cdot,t)\big)& \textrm{ for a.e.~}t>0\,, \label{eq:RescD-2}\\
  u_m &\geq 0\,\quad 
  0\leq \xi_m\leq 1\,,\quad 
  u_m\xi_m = u_m & \textrm{ a.e~on } \Gamma\times (0,\infty)\,, \label{eq:RescD-4}
  \\
  u_m(\cdot,0) &= u_{0,m}& \textrm{ on } \Gamma\,, \label{eq:RescD-5}
\end{align}
see Remark \ref{rem:Lh} for the definition and some properties of the family of operators $L_h$.
For general $D>0$ the only change would be that $w_m(\cdot,t)=\tfrac{1}{D}L_{\tfrac{g}{D}}\big((1-g)\xi_m(\cdot,t)\big)$.

Again, the mass conservation property \eqref{eq:RescMassCons} holds.

For the solutions $(u_m,\xi_m,w_m)_{m>0}$ to a family of initial data satisfying $\int_\Gamma u_{0,m}\,d\sigma=1$ for all $m>0$ and \eqref{eq:RescAssUm0} we obtain in the limit $m\to\infty$ a bulk--surface system.
\begin{lemma}[Limit system]
With $m\to 0$ the solutions $(u_m)_{m>0}$ converge in $L^2_{\loc}(\Gamma\times (0,\infty))$ to the unique solution $(u,\xi,w)$ of
\begin{align}
  \partial_t u &= -(1-g)\xi + w g & \textrm{ on } \Gamma\times (0,\infty)\label{eq:65a}\\
  w(\cdot,t) &=L_g\big((1-g)\xi(\cdot,t)\big)& \textrm{ for a.e.~}t>0\,, \label{eq:66}\\
  u &\geq 0\,\quad 
  0\leq \xi\leq 1\,,\quad 
  u\xi = u & \textrm{ a.e~on } \Gamma\times (0,\infty)\,, \label{eq:65}
  \\
  u(\cdot,0) &= u_0& \textrm{ on } \Gamma\,. \label{eq:66b}
\end{align}
\end{lemma}

\begin{proof}
Choose any $T>0$. It follows as in Lemma \ref{L.aprioriwk} that 
\begin{equation}
  \alpha_*\leq w_m\leq \alpha^*\,.
  \label{eq:BoundsW}
\end{equation}

As in Lemma \ref{lem:LimSys} we deduce uniform bounds
\begin{align}
  \|u_m\|_{L^\infty(\Gamma_T)}+\|u_m\|_{L^\infty(0,T;H^1(\Gamma))} +\sqrt m\|\Delta u_m\|_{L^2(\Gamma_T)}+\|\partial_t u_m\|_{L^2(\Gamma_T)}\leq C(T)\,.
\end{align}
We conclude that there exists a subsequence $m_k\to 0$ ($k\to\infty$), and functions $u\in H^1(\Gamma_T)$, $\xi\in L^\infty(\Gamma_T)$, $w\in L^\infty(\Gamma_T)$ such that
\begin{align}
  u_{m_k}&\to u \quad\textrm{ in }L^2(\Gamma_T)\,,\\
  u_{m_k}&\weakto u \quad\textrm{ in }H^1(\Gamma_T)\,\\
  \xi_{m_k}&\weakstarto \xi \quad\textrm{ in }L^\infty(\Gamma_T)\,,\\
  w_{m_k}&\weakstarto w \quad\textrm{ in }L^\infty(\Gamma_T)\,.
\end{align}

By a straightforward adaptation of the arguments in the proof of Lemma \ref{lem:LimSys} we deduce that \eqref{eq:65a} and \eqref{eq:65} hold.
By \eqref{eq:LhCont} and an identification argument we deduce that \eqref{eq:66} holds.

Following the proof of \cite[Theorem 4.1]{LNRV21} we obtain an $L^1$-contraction property for solutions to \eqref{eq:65a}-\eqref{eq:66b}, which implies uniqueness of solutions.
By a standard identification and continuation argument this shows the convergence of the whole sequence towards the unique solution of the limit system.
\end{proof}

\begin{proposition}[Monotonicity of support]
\label{4prop:Mono-spt}
For almost all $0\leq t_0<t$ we have
\begin{equation*}
  u(\cdot,t)=0 \quad\textrm{ a.e.~in }\{u(\cdot,t_0)=0\}\,.
\end{equation*}
\end{proposition}

\begin{proof}
For a set $B\subset [0,\infty)$ of measure zero and all $t_0\in [0,\infty)\setminus B$ we have
\begin{equation*}
  -(1-g)\xi(\cdot,t_0) + gL_g\big((1-g)\xi\big)(\cdot,t_0) =0 \quad\textrm{ in }\{u(\cdot,t_0)=0\}\,.
\end{equation*}
We fix $t_0\in [0,\infty)\setminus B$ and assume without loss of generality $t_0=0$.
We define $\Chi(\cdot,t)=\Chi_{\{u(\cdot,t)>0\}}$, $\Chi_0=\Chi(\cdot,0)$, $\xi_0=\xi(\cdot,0)$, and $\Chi^*(\cdot,t)=\Chi_{\{u_0+u(\cdot,t)>0\}}$.
By \cite{LNRV21} it holds
\begin{equation*}
  L_g\big((1-g)\xi\big)=L_{g\Chi^*}\big((1-g)\Chi^*\xi\big)\,,\quad
  L_g\big((1-g)\xi_0\big)=L_{g\Chi^*}\big((1-g)\Chi^*\xi_0\big)\,,\quad L_{g\Chi^*}\big(g\Chi^*\big)=1\,.
\end{equation*}
Therefore, we deduce for almost all $t\in t_0\in (0,\infty)\setminus B$
\begin{align*}
  &\frac{d}{dt}\int_\Gamma (1-\Chi_0)u(\cdot,t)\,d\sigma\\
  &\quad= \int_\Gamma (1-\Chi_0)\Chi(\cdot,t)\Big(-(1-g) + gL_g\big((1-g)\xi\big)\Big)\,d\sigma\\
  &\quad= \int_\Gamma (1-\Chi_0)\Chi(\cdot,t)\Big(-(1-g)(1-\xi_0) + gL_{g\Chi^*(\cdot,t)}\big((1-g)\Chi^*(\cdot,t)(\xi(\cdot,t)-\xi_0)\big)\Big)\,d\sigma\\
  &\quad= -\int_\Gamma (1-\Chi_0)\Chi(\cdot,t)(1-\xi_0)(1-g)L_{g\Chi^*(\cdot,t)}\big(g\Chi^*(\cdot,t)\big)\,d\sigma\\
  &\quad\qquad + \int_\Gamma \big((1-g)\Chi^*(\cdot,t)\big)(\xi(\cdot,t)-\xi_0)L_{g\Chi^*(\cdot,t)}\big((1-\Chi_0)\Chi(\cdot,t) g\big)\,d\sigma\,,
\end{align*}
where we have used that $L_h$ is self-adjoint.
Since $\Chi^*(\cdot,t)=\Chi_0 + (1-\Chi_0)\Chi(\cdot,t)$ we obtain
\begin{align*}
  \frac{d}{dt}\int_\Gamma (1-\Chi_0)u(\cdot,t)\,d\sigma
  &= -\int_\Gamma (1-\Chi_0)(1-\xi_0)\Chi(\cdot,t)(1-g)L_{g\Chi^*(\cdot,t)}\big(g\Chi_0\big)\,d\sigma \\
  &\qquad -\int_\Gamma \big((1-g)\Chi_0(1-\xi(\cdot,t))\big)L_{g\Chi^*(\cdot,t)}\big((1-\Chi_0)\Chi(\cdot,t) g\big)\,d\sigma\leq 0\,.
\end{align*}
Since $(1-\Chi_0)u(\cdot,0)=0$ we deduce
\begin{equation*}
  \int_{\{u(\cdot,t_0)=0\}}u(\cdot,t) \,d\sigma=0 \qquad\textrm{ for almost all }t>t_0\,,
\end{equation*}
which proves the claimed monotonicity property.
\end{proof}

As for the reduced system we obtain a representation formula for the solution $u$.
\begin{corollary}[Representation formula]
For any solution $(u,\xi,w)$ of \eqref{eq:65a}-\eqref{eq:66b} we have for almost all $t>0$
\begin{equation*}
  u(\cdot,t) = \Big(u_0 + \int_0^t -(1-g)(\cdot,\tau)+w(\cdot,\tau)g\,d\tau\Big)_+\,.
\end{equation*}
In particular, $u:\Gamma\times[0,\infty)$ has a continuous representative.
\end{corollary}
\begin{proof}
This follows as above, using in addition that $w\in L^\infty\big(0,\infty,C^{0,1/2}(\Gamma)\big)$.
\end{proof}

\begin{corollary}[Monotonicity of $w$ and $\xi$]
For almost all $0\leq t_0<t$ we have
\begin{align*}
  w(\cdot,t) &\leq w(\cdot,t_0)\quad\textrm{ a.e.~in }\Gamma\,\\
  \xi(\cdot,t) &\leq \xi(\cdot,t_0)\,.
\end{align*}
\end{corollary}

\begin{proof}
For $t_0<t \in [0,\infty)\setminus B$ as in the proof of Proposition \ref{4prop:Mono-spt} we deduce by $\{u(\cdot,t)>0\}\subset \{u(\cdot,t_0)>0\}$, by $\xi\leq 1$, and by the monotonicity of $L_h$
\begin{equation*}
  w(\cdot,t)= L_{g\Chi(\cdot,t_0)}(\Chi(\cdot,t_0)(1-g)\xi(\cdot,t))
  \leq L_{g\Chi(\cdot,t_0)}(\Chi(\cdot,t_0)(1-g)) = w(\cdot,t_0)\,.
\end{equation*}
The monotonicity of $\xi$ follows as in the proof of Corollary \ref{Cor:Mon}.
\end{proof}

\begin{theorem}[Concentration of support]
It holds 
\[
 w(\cdot,t)\searrow \alpha_* \quad\textrm{ in }C^{0,\frac{1}{2}}(\Gamma)
\]
and there exists a Borel set $S\subset\Gamma$ such that
\begin{equation}
  S=\mathsmaller\bigcap_{t>0} S(t) \qquad \mbox{ with } S(t):=\{u(\cdot,t)>0\}\,.
  \label{eq:LimSWInfty}
\end{equation}
Then $S\subset \{g=\gmax\}$ holds and for any sequence $t_k\to\infty$ there exists a subsequence (not relabeled) and a Radon measure $\mu$ on $\Gamma$ with support in $S$ such that
\begin{equation}
  \mu = \lim_{k\to\infty} u(\cdot,t_k)\calH^2 \quad\textrm{ in the sense of Radon measures on }\Gamma\,.
  \label{eq:CompactUT2}
\end{equation}
\end{theorem}

\begin{proof}
We follow the arguments and steps of the proof of Theorem \ref{thm:Conc1}.
\begin{steps}
\item Using in addition \eqref{eq:BoundsW} we deduce as before the existence and required properties of $S$, the compactness property \eqref{eq:CompactUT2} and the existence of limits $w_\infty,\xi_\infty$ of $w(\cdot,t_k),\xi(\cdot,t_k)$ in every $L^p(\Gamma)$, $1\leq p<\infty$.
By \eqref{eq:LhCont} we even have
\begin{equation}
  w(\cdot,t_k)\to w_\infty= L_g((1-g)\xi_\infty)\quad\textrm{ uniformly.}
  \label{eq:ConvWUni}
\end{equation} 
\item As before we prove that
\begin{equation}
  (1-g)\xi_\infty=g w_\infty \quad\textrm{ a.e.~in }\mathsmaller\bigcup\limits_{t>0}\{u(\cdot,t)=0\}=\Gamma\setminus S\,.
\label{eq:D1001}
\end{equation}
\item Using \eqref{eq:ConvWUni} we obtain as in the case $D=\infty$ that
\begin{equation}
  w_\infty \geq \frac{1-g}{g} \geq \alpha_*\quad\textrm{ in }S\,.
  \label{eq:WInfty}
\end{equation}
\item Similarly, we also deduce that 
\begin{equation}
  w_\infty = \frac{1-g}{g} \quad\calH^2-\textrm{a.e.~in }S\,.
\label{eq:aWInftyIdent}
\end{equation}
\item By \eqref{eq:D1001} and \eqref{eq:aWInftyIdent} we obtain $w_\infty \leq \frac{1-g}{g}$ in $\Gamma$, hence $w_\infty\leq\alpha_*$.
With \eqref{eq:WInfty} we deduce $w_\infty=\alpha_*$ in $S$.

Using Remark \ref{rem:Lh} we finally have
\begin{equation*}
  w_\infty=L_{g\Chi_S}\big((1-g)\Chi_S\big)=\alpha_*L_{g\Chi_S}\big(g\Chi_S\big)=\alpha_*\,.
\end{equation*}
\end{steps}
\end{proof}

\begin{theorem}[Distribution of limit measures]\label{thm:ConcD2}
Consider $S$, $w_\infty$ as in \eqref{eq:LimSWInfty} and assume $\calH^2(S)=0$.
Let a sequence $t_k\to\infty$ and a Radon measure $\mu$ on $\Gamma$ with support in $S$ be given such that
\begin{equation}
  \mu = \lim_{k\to\infty} u(\cdot,t_k)\calH^2 \quad\textrm{ in the sense of Radon measures on }\Gamma\,.
\end{equation}
Then there exists $\omega(k)\to 0$ for $k\to\infty$ such that for any measurable set $A\subset \Gamma$ with $\mu(A)>0$
\begin{equation*}
  \mu(A) = \lim_{k\to\infty}
  \frac{
    \int_A
      \big(g-(\gmax-\omega(k))\big)_+\,d\sigma
  }{
    \int_\Gamma
      \big(g-(\gmax-\omega(k))\big)_+\,d\sigma
  }\,.
\end{equation*}
\end{theorem}

\begin{proof}

We set
\begin{equation*}
  \varphi := (w-w_\infty)\,\gmax\,,\qquad
  h := \frac{\gmax-g}{\gmax}\,.
\end{equation*}
Then $g = (1-h)\,\gmax$ and $w = \frac{\varphi}{\gmax} + w_\infty$.

Consequently, from \eqref{eq:65a}-\eqref{eq:66b} we obtain
\begin{align*}
  \partial_t u &= -(h-\varphi+h\varphi)\,\chi_{\{u>0\}} \quad\textrm{ in }\Gamma\times (0,\infty)\,,\\
  \varphi &= \gmax\Big(L_{g\chi_{\{u>0\}}}\big((1-g)\chi_{\{u>0\}}\big)-w_\infty\Big)\,\\
  u &\geq 0\,,\quad 0 \le \xi \le 1\,,\quad u\,\xi = u\,.
\end{align*}

We now can follow the proof of Theorem \ref{thm:Conc2}.
The main difference is the additional space dependence of $\varphi$ but by the uniform convergence of $\varphi(\cdot,t)\to 0$ with $t\to\infty$ the arguments all remain valid.
\end{proof}

\section{Examples for the limit distribution of the total mass}\label{S.examples}

With $h:=\gmax-g$ the characterization of any limit measure $\mu$ with \eqref{eq:LimMu} in Theorem \ref{thm:Conc2} can be written as
\begin{equation*}
  \mu(A) = \lim_{k\to\infty}
  \frac{
    \int_A
      \big(\omega(k)-h\big)_+\dd S
  }{
    \int_\Gamma
      \big(\omega(k)-h)\big)_+\dd S
  }\quad\textrm{ for $A\subset\Gamma$ measurable}\,.
\end{equation*}
In this section we compute this limit measure for some examples. For that we assume in the following that $h:\Gamma\to [0,\infty)$ is continuous with 
\begin{equation*}
  \{h=0\} = \{x_1,\dots,x_M\}\subset\Gamma
\end{equation*}
for some $M\in\N$.
We choose $r_0>0$ such that the (intrinsic) balls $B(x_j,r_0)\subset\Gamma$, $j=1,\dots,M$, are pairwise disjoint.

We then define $G,G_j,f_j:(0,1]\to [0,\infty)$, $j=1,\dots,M$, by
\begin{align*}
  G(s) &:= \int_\Gamma (s-h(x))_+\,d\sigma(x)\,,\quad
  G_j(s) := \int_{B(x_j,r_0)} (s-h(x))_+\,d\sigma(x)\,,\quad
  f_j(s) := \frac{G_j(s)}{G(s)} \in [0,1]\,.
\end{align*}
Then  $f_j(s)$ is the fraction of the total mass of $(s-h)_+$ that lies
inside the unit ball around the zero $x_j$.

The asymptotic behavior of $f_1,\dots,f_M$ for $s\downarrow 0$ is related to the limit measures $\mu$ that were considered in Theorem \ref{thm:Conc2} and reflects the fraction of active protein concentrating in the points $x_1, \dots, x_M$. 

Since $\Gamma$ is $C^2$-regular and since the support of $(s-h)_+$ in $B(x_j,r_0)$ shrinks to $x_j$ with $s\downarrow 0$ it will be sufficient to restrict to the case of continuous functions  
 $h:\R^2\to [0,\infty)$ that satisfy $h \geq c_0>0$ on $\R^2\backslash \cup_j B(x_j,r_0)$.

\paragraph{Asymptotically homogeneous functions.}
We first consider the  case that 
\[
 \frac{h(x)}{u_j (x-x_j)} \to 1 \quad\mbox{ as } x \to x_j
\]
for some positively homogeneous function $u_j$ of degree $m_j>0$, that is 
$u_j(\lambda x) = \lambda^{m_j} u_j(x)$ for all $\lambda >0$ and $x \in \R^2$, $u_j(0)=0$ and $u_j(x)>0$ for all $x \not=0$.

Then we compute for sufficiently small $s$
\begin{align}\label{Gjasymp}
 G_j(s) &= \int_{B(x_j,r_0)} \big(s- u_j (x-x_j) + o(|x-x_j|^{m_j}) \big)_+ \,dx \nonumber \\
 &= s \int_{B(x_j,r_0)} \Big ( 1- \frac{u_j(x-x_j)}{s} + \frac{o(|x-x_j|^{m_j})}{s}\Big)_+\,dx \nonumber\\
 & = s \int_{B(x_j,r_0)} \Big( 1 - u_j \Big( \frac{x-x_j}{s^{\frac{1}{m_j}} } \Big) + \frac{o(|x-x_j|^{m_j})}{s}\Big)_+\,dx \nonumber\\
 &= s^{1+\frac{2}{m_j}} \int_{\{u_j<1\}} \big( 1-u_j(y)\big)\,dy + o\big(s^{1+\frac{2}{m_j}}\big) \nonumber\\
 &=: C_j s^{1+\frac{2}{m_j}} + o\big(s^{1+\frac{2}{m_j}}\big)\, 
\end{align}
with 
\begin{equation}\label{cjdef}
 C_j=\int_{\{u_j<1\}} \big( 1-u_j(y)\big)\,dy
\end{equation}

Let
\[
m_* := \max_{1\le j\le M} m_j,\qquad
J_* := \{j : m_j = m_*\}.
\]
Then, by \eqref{Gjasymp} we find 
\[
G_j(s) =
\begin{cases}
C_j\, s^{1+2/m_*} + o\bigl(s^{1+2/m_*}\bigr), & j\in J_*,\\[3pt]
o\bigl(s^{1+2/m_*}\bigr), & j\notin J_*,
\end{cases}
\]
as $s\downarrow0$. Hence
\[
G(s) = \sum_{j=1}^M G_j(s)
     = \Bigl(\sum_{j\in J_*} C_j\Bigr)s^{1+2/m_*}
       + o\bigl(s^{1+2/m_*}\bigr).
\]
As as consequence it follows as  $s\downarrow0$ that 
\[ \mbox{ if } m_j<m_*\,, \; \mbox{shen } f_j(s)\to0\,, \qquad \mbox{ whereas if } 
  m_j=m_*\,, \; \mbox{ then } 
          f_j(s) \to \frac{C_j}{\sum_{k\in J_*} C_k} =:w_j.
        \]
Thus the limiting weights $(w_1,\dots,w_M)$ form a probability vector
supported on the set $J_*$ of zeros at which $h$ is flattest.

A particularly simple case in this setting is $u_j(x) =\beta_j|x|^{m_j}$ for some $\beta_j>0$, where $C_j$ in \eqref{cjdef} is given by 
\begin{equation}\label{cjsimple}
 C_j= 2 \pi \int_0^1 \big( 1- \beta_j r^{m_j}\big)r\,dr = \pi \beta_j^{-\frac{2}{m_j}}  \frac{m_j}{2+m_j}\,.
\end{equation}

\paragraph{Nondegenerate minima.}
An important special case is when $h$ is $C^2$ near each minimum and
each zero is a non-degenerate quadratic minimum. More precisely, if we assume 
that for each $j$  the Hessian $H_j := D^2 h(x_j)$ is positive definite, then  $h$ is asymptotically homogeneous near each $x_j$ with $m_j=2$ and $u_j(x) = \frac 1 2 |\sqrt{H_j}x|^2$. Then, with the change of variables $y = \sqrt{H_j} x$ we obtain  
\begin{align*}
C_j = \int_{\{u_j <1\}} \big( 1- \tfrac 1 2 |\sqrt{H_j}x|^2\big)\,dx = \frac{1}{\sqrt{ \det H_j}} \int_{|y|^2<2} \big(1- \tfrac 1 2 |y|^2\big)\,dy  =\frac{\pi}{(\det H_j)^{1/2}} 
\end{align*}
and it follows that 
\[
\lim_{s\to 0} f_j(s)
  = \frac{(\det H_j)^{-1/2}}{\sum_{k=1}^M (\det H_k)^{-1/2}}.
\]
Hence, when all minima are non-degenerate quadratic, {all} of them
contribute in the limit and the flatter minima attract
a larger fraction of the mass of $(s-h)_+$ as $s\to 0$.

\paragraph{An oscillatory example.}
 
 With the insight gained for the asymptotically homogenous functions $h$ we can now also construct an example for a function $h$ with a single minimum at, say $x=0$, such that $G(s)$ oscillates as $t \to 0$. More precisely, we construct $h$ such that there exists two sequences $\{s_k\}_k, \{\tilde s_k\}_k$ with $s_k \to 0$ and $\tilde s_k \to 0$ and $G(s_k) \sim \big(s_k\big)^a$ and 
 $G(\tilde s_k) \sim \big(\tilde s_k\big)^b$ as $k \to \infty$ with $a>b$.
 
 For the construction we first state the following Lemma.
 \begin{lemma}\label{L.oscill}
  Let $a>b>0$ and $\theta>1$ be given. Then there exists an increasing function $\psi \colon [0,\infty) \to [0,\infty)$ with $\psi(0)=0$ and decreasing sequences 
  $\{r_k\}_k, \{ \rho_k\}_k, \{\tilde r_k\}_k, \{ \tilde \rho_k\}_k$ such that 
  \[
   \rho_k \leq (r_k)^{\theta}\,, \quad \tilde \rho_k \leq \big(\tilde r_k\big)^{\theta} 
\quad \mbox{ and } \quad \tilde r_k < \rho_k\,,\, r_{k+1} < \tilde \rho_k
\]
such that 
\[
 \psi(r) = r^a  \quad \mbox{ in } [\rho_k,r_k]\quad \mbox{ and } \quad 
 \psi(r)=r^b  \quad \mbox{ in } [\tilde \rho_k, \tilde r_k]\,.
\]
 \end{lemma}
\begin{proof}
 Let $r_1<1$ and set $\psi(r)=r^a$ in $[\rho_1,r_1]$ for some $0<\rho_1<r_1$ with $\rho_1\leq r_1^\theta$. 
 Then choose $\tilde r_1$ with $(\tilde r_1)^b < (\rho_1)^a$ and choose $\psi(r)$ increasing (e.g. as an affine function) in $[\tilde r_1,\rho_1]$ with 
 $\psi(\tilde r_1) = (\tilde r_1)^b$ and $\psi(\rho_1)= (\rho_1)^a$. Next choose $\tilde\rho_1<\tilde r_1^\theta$ and set $\psi(r)=r^b$ in $[\tilde \rho_1,\tilde r_1]$ and increasing in $[r_2,\tilde r_1]$ for $r_2 < \tilde \rho_1$ with $\psi(r_2)=r_2^a$. Iterating this procedure yields the desired function.
\end{proof}

Now we choose $a>b>0$ and $\theta>1$ and set $h(x)=\psi(|x|)$ with $\psi$ as in Lemma \ref{L.oscill}. Furthermore we set $s_k=r_k^a$ and $\tilde s_k= \big(\tilde r_k\big)^b$. Then 
\begin{align*}
 G(s_k)&= \int_{\R^2} \big(s_k-h(x)\big)_+\,dx\\
 &= \int_{\rho_k\leq |x| \leq r_k} \big(s_k-|x|^a\big)_+\,dx + \int_{|x| \leq \rho_k} 
\big(s_k - h(x)\big)\,dx\\
&= \int_{|x| \leq r_k} \big(s_k-|x|^a\big)\,dx - \int_{|x| \leq \rho_k} \big(s_k-|x|^a\big)\,dx + \int_{|x| \leq \rho_k} \big(s_k-h(x)\big)\,dx\\
&= \frac{\pi a}{2+a} s_k^{1 + \frac{2}{a}} + \underbrace{\int_{|x| \leq\rho_k} \big( |x|^a-h(x)\big)\,dx}_{=:\eps_k}\,.
\end{align*}
The term $\eps_k$ can be estimated as 
\[
 |\eps_k| \leq \pi \rho_k^{2+a} \leq \pi r_k^{\theta (2+a)} \leq \pi s_k^{\theta (1+ \frac{2}{a})}
\]
and for $\theta>1$ we obtain that $\eps_k$ is of order $o\big(s_k^{1+\frac{2}{a}}\big)$. Hence 
\[
 \frac{G(s_k)}{s_k^{1+\frac{2}{a}}} \to \frac{\pi a}{2+a} \qquad \mbox{ as } k \to \infty
\]
and similarly we find 
\[
 \frac{G(\tilde s_k)}{{\tilde s_k}^{1+\frac{2}{b}}} \to \frac{\pi b}{2+b} \qquad \mbox{ as } k \to \infty\,.
\]

\paragraph{Oscillating masses.}
With the example from the previous paragraph we can now also easily construct a function $h$ with two minima such that the mass oscillates between the two minima as $t\to 0$.

For that assume that $h$ has two zeros in $x_1=0$ and in $x_2$ with $|x_1-x_2| \geq 3$. Then we choose $d \in (b,a)$ where $a>b>0$ and set $h(x)=\psi(|x|)$ in $B(0,1)$ with $\psi$ as above. Furthermore we set $h(x)=|x-x_2|^d$ in $B(x_2,1)$ and otherwise smooth and larger than one outside the two unit balls around $x_1$ and $x_2$. 

For $t<1$ we have $G(s)=G_1(s)+G_2(s)$ with   (cf. \eqref{cjsimple})
\[
 G_2(s) \sim  \frac{\pi d}{2+d} s^{1+ \frac{2}{d}} \qquad \mbox{ as } t \to 0
\]
and $G_1(s)$ behaves as $G(s)$ in the example above. Thus, in this case with the sequences $\{s_k\}_k, \{\tilde s_k\}_k$ as above
\[
 G(s_k) \sim \frac{\pi a}{2+a} s_k^{1+\frac{2}{a}} \qquad \mbox{ and } \qquad 
 G(\tilde s_k) \sim \frac{\pi d}{2+d} \big( \tilde s_k\big)^{1+\frac{2}{d}} \qquad \mbox{ as } k \to \infty\,.
\]
Consequently 
\[
 f_1(s_k) \to 1 \qquad \mbox{ and } \qquad  f_2(s_k) \to 0 \qquad \mbox{ as } k \to \infty\,,
\]
while 
\[
 f_1(\tilde s_k) \to 0 \qquad \mbox{ and } \qquad f_2(\tilde s_k) \to 1 \qquad \mbox{ as } k \to \infty\,. 
\]

\bigskip
{\bf Acknowledgement}

The authors gratefully acknowledge  support of the Deutsche Forschungsgemeinschaft (DFG,
German Research Foundation) through the collaborative research centre “Analysis of criticality: from
complex phenomena to models and estimates” (CRC 1720, Project-ID 539309657).
MR also acknowledges the hospitality of the Institute of Applied Mathematics at the University of Bonn.

\printbibliography

@article{GK87,
 author = {Giga, Yoshikazu and Kohn, Robert V.},
 title = {Characterizing blowup using similarity variables},
 fjournal = {Indiana University Mathematics Journal},
 journal = {Indiana Univ. Math. J.},
 issn = {0022-2518},
 volume = {36},
 pages = {1--40},
 year = {1987},
 language = {English},
 doi = {10.1512/iumj.1987.36.36001},
 zbMATH = {3969175},
 Zbl = {0601.35052}
}

@article{GK85,
 author = {Giga, Yoshikazu and Kohn, Robert V.},
 title = {Asymptotically self-similar blow-up of semilinear heat equations},
 fjournal = {Communications on Pure and Applied Mathematics},
 journal = {Commun. Pure Appl. Math.},
 issn = {0010-3640},
 volume = {38},
 pages = {297--319},
 year = {1985},
 language = {English},
 doi = {10.1002/cpa.3160380304},
 zbMATH = {3937666},
 Zbl = {0585.35051}
}

@article{AbKa20,
  title={On a model for phase separation on biological membranes and its relation to the Ohta–Kawasaki equation},
  volume={31},
  DOI={10.1017/S0956792519000056},
  number={2},
  journal={European Journal of Applied Mathematics},
  author={Abels, H. and Kampmann, J.},
  year={2020},
  pages={297–338}
}

@article{MoJE08,
 author = {Mori, Yoichiro and Jilkine, Alexandra and Edelstein-Keshet, Leah},
 doi = {10.1529/biophysj.107.120824},
 institution = {Institute of Applied Mathematics and Department of Mathematics University of British Columbia, Vancouver, Canada.},
 issn = {0006-3495},
 journal = {Biophys. J.},
 language = {eng},
 medline-pst = {ppublish},
 number = {9},
 pages = {3684--3697},
 pii = {S0006-3495(08)70444-2},
 pmid = {18212014},
 publisher = {Elsevier BV},
 source = {Crossref},
 title = {{Wave-Pinning and Cell Polarity from a Bistable Reaction-Diffusion System}},
 url = {https://doi.org/10.1529/biophysj.107.120824},
 volume = {94},
 year = {2008}
}

@article{GKRR16,
 author = {Garcke, Harald and Kampmann, Johannes and Rätz, Andreas and Röger, Matthias},
 doi = {10.1142/s0218202516500275},
 fjournal = {Mathematical Models and Methods in Applied Sciences},
 issn = {0218-2025, 1793-6314},
 journal = {Math. Models Methods Appl. Sci.},
 mrclass = {35K51 (35Q92 35R37 92C37)},
 mrnumber = {3484571},
 number = {06},
 pages = {1149--1189},
 publisher = {World Scientific Pub Co Pte Lt},
 source = {Crossref},
 title = {{A coupled surface-{Cahn–Hilliard} bulk-diffusion system modeling lipid raft formation in cell membranes}},
 url = {https://doi.org/10.1142/s0218202516500275},
 volume = {26},
 year = {2016}
}

@article{MFNS23,
  author = {Miller, Pearson W. and Fortunato, Daniel and Novaga, Matteo and Shvartsman, Stanislav Y. and Muratov, Cyrill B.},
  title = {Generation and Motion of Interfaces in a Mass-Conserving Reaction-Diffusion System},
  journal = {SIAM Journal on Applied Dynamical Systems},
  volume = {22},
  number = {3},
  pages = {2408-2431},
  year = {2023},
  doi = {10.1137/22M152548X},
  URL = {https://doi.org/10.1137/22M152548X},
  eprint = {https://doi.org/10.1137/22M152548X},
}

@article{LeIg02,
 author = {Levchenko, Andre and Iglesias, Pablo A.},
 doi = {10.1016/s0006-3495(02)75373-3},
 issn = {0006-3495},
 journal = {Biophys. J.},
 number = {1},
 pages = {50--63},
 publisher = {Elsevier BV},
 source = {Crossref},
 title = {{Models of Eukaryotic Gradient Sensing: {Application} to Chemotaxis of Amoebae and Neutrophils}},
 url = {https://doi.org/10.1016/s0006-3495(02)75373-3},
 volume = {82},
 year = {2002}
}

@article{PaDe99,
 author = {Parent, Carole A. and Devreotes, Peter N.},
 doi = {10.1126/science.284.5415.765},
 issn = {0036-8075, 1095-9203},
 journal = {Science},
 number = {5415},
 pages = {765--770},
 publisher = {American Association for the Advancement of Science (AAAS)},
 source = {Crossref},
 title = {{A Cell's Sense of Direction}},
 url = {https://doi.org/10.1126/science.284.5415.765},
 volume = {284},
 year = {1999}
}

@article{Turi52,
 author = {Turing, Alan M.},
 copyright = {Copyright Â 1952 The Royal Society},
 doi = {10.1098/rstb.1952.0012},
 issn = {2054-0280},
 journal = {Philos. Trans. R. Soc. Lond. B. Biol. Sci.},
 jstor_articletype = {primary
_article},
 jstor_formatteddate = {Aug. 14, 1952},
 number = {641},
 pages = {37--72},
 publisher = {The Royal Society},
 source = {Crossref},
 title = {{The chemical basis of morphogenesis}},
 url = {https://doi.org/10.1098/rstb.1952.0012},
 volume = {237},
 year = {1952}
}

@article{NGCR07,
 author = {Novak, Igor L. and Gao, Fei and Choi, Yung-Sze and Resasco, Diana and Schaff, James C. and Slepchenko, Boris M.},
 doi = {10.1016/j.jcp.2007.05.025},
 institution = {Center for Cell Analysis and Modeling, Department of Cell Biology, University of Connecticut Health Center, Farmington, Connecticut 06030.},
 issn = {0021-9991},
 journal = {J. Comput. Phys.},
 language = {eng},
 medline-pst = {ppublish},
 number = {2},
 pages = {1271--1290},
 pmid = {18836520},
 publisher = {Elsevier BV},
 source = {Crossref},
 title = {{Diffusion on a curved surface coupled to diffusion in the volume: {Application} to cell biology}},
 url = {https://doi.org/10.1016/j.jcp.2007.05.025},
 volume = {226},
 year = {2007}
}

@article{Miel13,
 author = {Mielke, Alexander},
 doi = {10.3934/dcdss.2013.6.479},
 issn = {1937-1179},
 journal = {Discrete \& Continuous Dynamical Systems - S},
 mrnumber = {2997555},
 number = {2},
 pages = {479--499},
 publisher = {American Institute of Mathematical Sciences (AIMS)},
 source = {Crossref},
 title = {{Thermomechanical modeling of energy-reaction-diffusion systems, including bulk-interface interactions}},
 url = {https://doi.org/10.3934/dcdss.2013.6.479},
 volume = {6},
 year = {2013}
}

@misc{KoSU26,
      title={Dynamics of Interfaces in the Two-Dimensional Wave-Pinning Model}, 
      author={Shunsuke Kobayashi and Koya Sakakibara and Taikei Uechi},
      year={2026},
      eprint={2601.04746},
      archivePrefix={arXiv},
      primaryClass={math.DS},
      url={https://arxiv.org/abs/2601.04746}, 
}

@article {AlET18,
    AUTHOR = {Alphonse, Amal and Elliott, Charles M. and Terra, Joana},
     TITLE = {A coupled ligand-receptor bulk-surface system on a moving
              domain: well posedness, regularity, and convergence to
              equilibrium},
   JOURNAL = {SIAM J. Math. Anal.},
  FJOURNAL = {SIAM Journal on Mathematical Analysis},
    VOLUME = {50},
      YEAR = {2018},
    NUMBER = {2},
     PAGES = {1544--1592},
      ISSN = {0036-1410,1095-7154},
   MRCLASS = {35K57 (35B30 35B65 35Q92 35R01 35R37 81Q80 92C37)},
  MRNUMBER = {3775132},
MRREVIEWER = {Patrick\ Guidotti},
       DOI = {10.1137/16M110808X},
       URL = {https://doi.org/10.1137/16M110808X},
}

@article {Diss21,
    AUTHOR = {Disser, Karoline},
     TITLE = {Global existence and uniqueness for a volume-surface
              reaction-nonlinear-diffusion system},
   JOURNAL = {Discrete Contin. Dyn. Syst. Ser. S},
  FJOURNAL = {Discrete and Continuous Dynamical Systems. Series S},
    VOLUME = {14},
      YEAR = {2021},
    NUMBER = {1},
     PAGES = {321--330},
      ISSN = {1937-1632,1937-1179},
   MRCLASS = {35K61 (35A01 35B45 35K57)},
  MRNUMBER = {4186214},
MRREVIEWER = {Dian\ K.\ Palagachev},
       DOI = {10.3934/dcdss.2020326},
       URL = {https://doi.org/10.3934/dcdss.2020326},
}

@article {AuBo24,
    AUTHOR = {Augner, Bj\"orn and Bothe, Dieter},
     TITLE = {Analysis of bulk-surface reaction-sorption-diffusion systems
              with {L}angmuir-type adsorption},
   JOURNAL = {J. Math. Pures Appl. (9)},
  FJOURNAL = {Journal de Math\'ematiques Pures et Appliqu\'ees. Neuvi\`eme
              S\'erie},
    VOLUME = {188},
      YEAR = {2024},
     PAGES = {215--272},
      ISSN = {0021-7824,1776-3371},
   MRCLASS = {35A01 (35D35 35K57 35Q92 58J35)},
  MRNUMBER = {4766800},
       DOI = {10.1016/j.matpur.2024.05.001},
       URL = {https://doi.org/10.1016/j.matpur.2024.05.001},
}

@article {MoTa23,
    AUTHOR = {Morgan, Jeff and Tang, Bao Quoc},
     TITLE = {Global well-posedness for volume-surface reaction-diffusion
              systems},
   JOURNAL = {Commun. Contemp. Math.},
  FJOURNAL = {Communications in Contemporary Mathematics},
    VOLUME = {25},
      YEAR = {2023},
    NUMBER = {4},
     PAGES = {Paper No. 2250002, 63},
      ISSN = {0219-1997,1793-6683},
   MRCLASS = {35A01 (35K57 35K58 35Q92)},
  MRNUMBER = {4572497},
MRREVIEWER = {Nicola\ Pintus},
       DOI = {10.1142/S021919972250002X},
       URL = {https://doi.org/10.1142/S021919972250002X},
}

@Article{HaRo18,
  author    = {Hausberg, Stephan and R{\"o}ger, Matthias},
  title     = {Well-posedness and fast-diffusion limit for a bulk--surface reaction--diffusion system},
  journal   = {Nonlinear Differential Equations and Applications NoDEA},
  year      = {2018},
  volume    = {25},
  number    = {3},
  pages     = {17},
  issn      = {1420-9004},
  day       = {27},
  doi       = {10.1007/s00030-018-0508-8},
  url       = {https://doi.org/10.1007/s00030-018-0508-8},
}

@Article{         RaeRoe12,
  title         = {Turing instabilities in a mathematical model for signaling
                  networks},
  author        = {R{\"a}tz, Andreas and R{\"o}ger, Matthias},
  journal       = {J. Math. Biol.},
  year          = {2012},
  number        = {6-7},
  pages         = {1215--1244},
  volume        = {65},
  coden         = {JMBLAJ},
  doi           = {10.1007/s00285-011-0495-4},
  fjournal      = {Journal of Mathematical Biology},
  issn          = {0303-6812},
  mrclass       = {Preliminary Data},
  mrnumber      = {2993944},
  url           = {http://dx.doi.org/10.1007/s00285-011-0495-4}
}

@Article{         RaeRoe14,
  title         = {Symmetry breaking in a bulk–surface reaction–diffusion
                  model for signalling networks},
  author        = {Andreas Rätz and Matthias Röger},
  journal       = {Nonlinearity},
  year          = {2014},
  number        = {8},
  pages         = {1805},
  volume        = {27},
}

@Article{GiMe72,
  author      = {A. Gierer and H. Meinhardt},
  title       = {A theory of biological pattern formation.},
  journal     = {Kybernetik},
  year        = {1972},
  volume      = {12},
  number      = {1},
  pages       = {30--39},
  language    = {eng},
  medline-pst = {ppublish},
  pmid        = {4663624},
}

@Article{RaEd17,
  author        = {Rappel, Wouter-Jan and Edelstein-Keshet, Leah},
  title         = {Mechanisms of Cell Polarization.},
  journal       = {Current opinion in systems biology},
  year          = {2017},
  volume        = {3},
  pages         = {43--53},
  month         = jun,
  issn          = {2452-3100},
  __markedentry = {[matthias:6]},
  abstract      = {Cell polarization is a key step in the migration, development, and organization of eukaryotic cells, both at the single cell and multicellular level. Research on the mechanisms that give rise to polarization of a given cell, and organization of polarity within a tissue has led to new understanding across cellular and developmental biology. In this review, we describe some of the history of theoretical and experimental aspects of the field, as well as some interesting questions and challenges for the future.},
  country       = {England},
  doi           = {10.1016/j.coisb.2017.03.005},
  issn-linking  = {2452-3100},
  mid           = {NIHMS867552},
  nlm-id        = {101698476},
  owner         = {NLM},
  pmc           = {PMC5640326},
  pmid          = {29038793},
  pubmodel      = {Print-Electronic},
  pubstatus     = {ppublish},
  revised       = {2018-11-13},
}

@Article{ElRV17,
  author    = {Elliott, Charles M and Ranner, Thomas and Venkataraman, Chandrasekhar},
  title     = {Coupled bulk-surface free boundary problems arising from a mathematical model of receptor-ligand dynamics},
  journal   = {SIAM Journal on Mathematical Analysis},
  year      = {2017},
  volume    = {49},
  number    = {1},
  pages     = {360--397},
  publisher = {SIAM},
}

@Article{DMMS18,
  author    = {Rocky Diegmiller and Hadrien Montanelli and Cyrill B. Muratov and Stanislav Y. Shvartsman},
  title     = {Spherical Caps in Cell Polarization},
  journal   = {Biophysical Journal},
  year      = {2018},
  volume    = {115},
  number    = {1},
  pages     = {26 - 30},
  issn      = {0006-3495},
  doi       = {https://doi.org/10.1016/j.bpj.2018.05.033},
  url       = {http://www.sciencedirect.com/science/article/pii/S0006349518306726},
}

@Article{BKMS17,
  author       = {Bothe, Dieter and K\"ohne, Matthias and Maier, Siegfried and Saal, J\"urgen},
  title        = {Global strong solutions for a class of heterogeneous catalysis models},
  volume       = {445},
  number       = {1},
  pages        = {677--709},
  issn         = {0022-247X},
  date         = {2017},
  doi          = {10.1016/j.jmaa.2016.08.016},
  fjournal     = {Journal of Mathematical Analysis and Applications},
  journaltitle = {J. Math. Anal. Appl.},
  mrclass      = {35Q92 (35D35 35K51)},
  mrnumber     = {3543789},
}

@Article{FeLT18,
  author    = {Fellner, Klemens and Latos, Evangelos and Tang, Bao Quoc},
  title     = {Well-posedness and exponential equilibration of a volume-surface reaction--diffusion system with nonlinear boundary coupling},
  journal   = {Ann. Inst. H. Poincar\'e Anal. Non Lin\'eaire},
  year      = {2018},
  volume    = {35},
  number    = {3},
  pages     = {643--673},
  issn      = {0294-1449},
  doi       = {10.1016/j.anihpc.2017.07.002},
  fjournal  = {Annales de l'Institut Henri Poincar\'e. Analyse Non Lin\'eaire},
  mrclass   = {35K61 (35A01 35B40 35K57)},
  mrnumber  = {3778646},
  url       = {https://doi.org/10.1016/j.anihpc.2017.07.002},
}

@Article{LeRa05,
  author    = {Levine, Herbert and Rappel, Wouter-Jan},
  title     = {Membrane-bound {T}uring patterns},
  journal   = {Phys. Rev. E},
  year      = {2005},
  volume    = {72},
  pages     = {061912},
  doi       = {10.1103/PhysRevE.72.061912},
  issue     = {6},
  numpages  = {5},
  publisher = {American Physical Society},
  url       = {http://link.aps.org/doi/10.1103/PhysRevE.72.061912},
}

@article{NiRV20,
 author = {Niethammer, Barbara and Röger, Matthias and Velázquez, Juan},
 doi = {10.4171/ifb/433},
 fjournal = {Interfaces and Free Boundaries. Mathematical Analysis, Computation and Applications},
 issn = {1463-9963},
 journal = {Interface. Free Bound.},
 mrclass = {92C37 (35K57)},
 mrnumber = {4089401},
 number = {1},
 pages = {85--117},
 publisher = {European Mathematical Society - EMS - Publishing House GmbH},
 source = {Crossref},
 title = {{A bulk-surface reaction-diffusion system for cell polarization}},
 url = {https://doi.org/10.4171/ifb/433},
 volume = {22},
 year = {2020}
}

@article{LNRV21,
 author = {Logioti, A. and Niethammer, B. and R{\"o}ger, M. and Vel{\'a}zquez, J. J. L.},
 URL = {https://doi.org/10.1137/20M1349114},
 eprint = {https://doi.org/10.1137/20M1349114},
 doi = {10.1137/20M1349114},
 fjournal = {SIAM Journal on Mathematical Analysis},
 issn = {0036-1410},
 journal = {SIAM J. Math. Anal.},
 language = {English},
 number = {1},
 pages = {1214--1238},
 title = {A parabolic free boundary problem arising in a model of cell polarization},
 volume = {53},
 year = {2021},
}

@article{LNRV23,
 author = {Anna Logioti and Barbara Niethammer and Matthias Röger and Juan J. L. Velázquez},
 doi = {10.1080/03605302.2023.2247467},
 eprint = {https://doi.org/10.1080/03605302.2023.2247467},
 journal = {Communications in Partial Differential Equations},
 number = {7-8},
 pages = {1-37},
 publisher = {Taylor \& Francis},
 title = {Qualitative properties of solutions to a mass-conserving free boundary problem modeling cell polarization},
 url = {https://doi.org/10.1080/03605302.2023.2247467},
 volume = {48},
 year = {2023}
}

@incollection {LNRV24,
    AUTHOR = {Logioti, Anna and Niethammer, Barbara and R\"oger, Matthias
              and Vel\'azquez, Juan J. L.},
     TITLE = {Interface behavior for the solutions of a mass conserving free
              boundary problem modelling cell polarization},
 BOOKTITLE = {Friends in partial differential equations---the {N}ina {N}.
              {U}raltseva 90th anniversary volume},
     PAGES = {249--274},
 PUBLISHER = {EMS Press, Berlin},
      YEAR = {2025},
      ISBN = {978-3-98547-094-5},
   MRCLASS = {35R37 (35L05 35R35)},
  MRNUMBER = {4967639},
}

@online{FNV26,
      title={On the shape of the positivity region for a free boundary problem describing cell polarization},
      author={Sebastián Flores Sepúlveda and Barbara Niethammer and Juan J. L. Velázquez},
      year={2026},
      eprint={2605.03553},
      archivePrefix={arXiv},
      primaryClass={math.AP},
      url={https://arxiv.org/abs/2605.03553},
}
\end{document}